\documentclass{article}
\usepackage{indentfirst,enumerate}
\usepackage{bm, mathrsfs, graphics,amssymb,amsmath,subeqnarray,setspace,graphicx,amsthm,epstopdf,subfigure,color}
\usepackage{ulem} 
\usepackage{float}

\usepackage{geometry}
\usepackage[colorlinks,
            linkcolor=red,
            anchorcolor=red,
            citecolor=red
            ]{hyperref}
            
\numberwithin{equation}{section}

\def\b{\boldsymbol}

\def\cP{\mathcal{P}}
\def\R{\mathbb{R}}
\def\N{\mathbb{N}}
\def\Gm{\Gamma}
\def\cS{\mathcal{S}}

\newcommand{\st}{\mbox{ s.t. }}

\newtheorem{theorem}{Theorem}[section]
\newtheorem{lemma}{Lemma}[section]
\newtheorem{proposition}{Proposition}[section]
\newtheorem{remark}{Remark}[section]

\newtheorem{definition}{Definition}[section]

\newcommand{\dif}{\mathrm{d}}

\newcommand{\bprop}{\begin{proposition}}
\newcommand{\eprop}{\end{proposition}}
\newcommand{\M}{\mathcal{M}(\Om)}

\newcommand{\one}{\mathbf{1}}
\newcommand{\inv}{\frac{1}}
\newcommand{\Om}{\Omega}
\newcommand{\x}{\boldsymbol{x}}
\newcommand{\seq}{\operatorname{seq}}
\newcommand{\ddiv}{\mathrm{div}^h}

\newcommand{\boldm}{\boldsymbol{m}}

\newcommand{\nlimsup}{\limsup_{n\rightarrow \infty}}
\newcommand{\ra}{\rightarrow}
\newcommand{\hra}{\Rightarrow}
\newcommand{\boldh}{\boldsymbol{h}}
\newcommand{\dm}{\mathbf{\delta m}}

\newcommand{\email}[1]{\protect\href{mailto:#1}{#1}}
\newcommand\funding[1]{\protect\\ \hspace*{15.37pt}{\bfseries Funding:} #1}

\title{On the convergence of continuous and discrete unbalanced optimal transport models\thanks{Submitted to the editors DATE.\funding{This work is partially supported by the National Key R\&D Program of China No. 2020YFA0712000 and No. 2021YFA1002800, and Shanghai Municipal Science and Technology Major Project (2021SHZDZX0102).
The work of Z. Xiong, Y. Zhu, X. Zhang was partially supported by NSFC (No.11771288; No.12090024).
The work of L. Li was partially supported by NSFC 11901389 and 12031013, and Shanghai Science and Technology Commission Grant No. 21JC1402900. We also thank the Student Innovation Center at Shanghai
Jiao Tong University for providing us the computing services.}}}
\author{Zhe Xiong\thanks{School of Mathematical Sciences, Shanghai Jiao Tong University, Shanghai 200240, China (\email{aristotle-x@sjtu.edu.cn},
\email{zynsjtu@sjtu.edu.cn}).}
\and Lei Li\thanks{Institute of Natural Sciences, Qing Yuan Research Institute, MOE-LSC, Shanghai Jiao Tong University, Shanghai 200240, China (\email{leili2010@sjtu.edu.cn},
\email{xqzhang@sjtu.edu.cn}).}
\and Ya-Nan Zhu\footnotemark[2]
\and Xiaoqun Zhang\footnotemark[3]
}

\begin{document}

\maketitle

\begin{abstract}
We consider a Beckmann formulation  of an unbalanced optimal transport (UOT) problem. The  $\Gamma$-convergence of this formulation of UOT  to  the corresponding optimal transport (OT) problem is established as the balancing parameter $\alpha$ goes to infinity. The discretization of the problem is further shown to be asymptotic preserving regarding the same limit, which ensures that a numerical method can be applied uniformly and the solutions converge to the one of the OT problem automatically. Particularly, there exists a critical value, which is independent of the mesh size, such that the discrete problem reduces to the discrete OT problem for $\alpha$ being larger than this critical value. The discrete problem is solved by a convergent primal-dual hybrid algorithm and the iterates for UOT are also shown to converge to that for OT. Finally, numerical experiments on shape deformation and partial color transfer are implemented to validate  the theoretical convergence and the proposed numerical algorithm.  
\end{abstract}

\section{Introduction}
The concept of optimal transport (OT) was first put forward in 1781 by Monge  \cite{monge1781memoire}  and was relaxed later by Kantorovich \cite{kantorovich2006problem} as a convex linear program. OT has then been extensively applied in various fields, including image processing \cite{ferradans2014regularized,papadakis2015optimal},   machine learning \cite{salimans2018improving,arjovsky2017wasserstein, frogner2015learning}, PDE theory \cite{trudinger2008monge,santambrogio2015optimal} and noise sampling \cite{de2012blue}. 
 We refer the readers to \cite{santambrogio2015optimal, villani2021topics,peyre2019computational}  for overviews of theoretic and computational optimal transport. 
The OT models have been extended to the so-called unbalanced optimal transport or unnormalized optimal transport (UOT) problems \cite{chizat2018unbalanced,figalli2010new,liero2018optimal,gangbo2019unnormalized} for applications involving mass distributions with different masses. Moreover, the UOT models can take into account of the weight change even for probability measures so that they can be used more flexibly \cite{chizat2018unbalanced,piccoli2014generalized}. For example, the UOT distance is applied to deal with the full waveform inverse problem \cite{li2022application} and used for waveform based earthquake location \cite{zhou2018wasserstein}. And in \cite{pmlr-v139-fatras21a} a gradient method based on UOT is put forward, which is employed in the domain adaption problem.


Let us start with the introduction to the OT problems. Suppose $X, Y$ are two topological spaces with probability measures $\nu_1, \nu_2$ respectively. Given a cost function $c: X \times Y \rightarrow \R^+$, the Kantorovich problem is to find a joint measure $\pi$ (called a ``transport plan") on the product space $X \times Y$ such that
\begin{equation}\label{Kantorovich}
    \begin{aligned}
        & \min_\pi \int_{X \times Y} c(x, y) \dif \pi(x, y),\\
        & \st~ \pi(A, Y) = \nu_1(A), \quad \pi(X, B) = \nu_2(B), \quad \text{for}~\forall A \subset X, B \subset Y.
        \end{aligned}
\end{equation}
In later discussions,  we only focus on the case $X=Y=\Omega\subset\R^d$ for a domain $\Omega$. Let $\cP(\Om)$ denote the set of probability measures on $\Om$ and $\mathcal{W}_p(\Om):=\{\mu\in \cP(\Om): \int |x|^p d\mu<\infty \}$. Choosing $c(x, y) = |x-y|^p$ for $p \geq 1$, then \eqref{Kantorovich} induces a  distance between two measures $\rho_0,\rho_1\in \mathcal{W}_p(\Om)$, which is the so-called Wasserstein-$p$ distance $W_p$
\begin{gather}
W_p(\rho_0, \rho_1)=\left(\inf_{\pi \in \Pi(\rho_0, \rho_1)}\int |x-y|^p d\pi\right)^{1/p},
\end{gather}
where $\Pi(\rho_0, \rho_1)$ is the set of all transport plans for $\rho_0$ and $\rho_1$. 
In the case of $p=1$, simplifying the dual problem of the Kantorovich formulation can lead to the following characterization of the $W_1$ distance
\begin{gather}\label{eq:w1characterization}
W_1(\rho_0, \rho_1)=\sup_{\varphi\in \mathrm{Lip}_1(X)}\int \varphi\, d(\rho_0-\rho_1).
\end{gather}
This characterization has important application in the generative models \cite{bernton2017inference, pmlr-v97-dukler19a}.
The dual problem of \eqref{eq:w1characterization} is given by the flow-minimization model introduced by Beckmann \cite[Section 4.2]{beckmann1952continuous,santambrogio2015optimal}:
\begin{equation}\label{eq:BeckmannOT}
 W_1(\rho_0, \rho_1) = \min \left\{\int_\Om \lvert \textbf{m} \rvert \dif x :~ \textbf{m}: \Om \to\R^d,~ \nabla \cdot \boldm=\rho_0 - \rho_1  \right\}.
\end{equation}
The Kantorovich problem mentioned in \eqref{Kantorovich} is often regarded as the static formulation. In \cite{benamou2000computational}, Benamou and Brenier proposed a dynamical version of optimal transport which seeks for a geodesic path between the two measures $\rho_0$ and $\rho_1$ when $\Omega$ is convex. Suppose $\M$ and $(\M)^d$ are the spaces of Radon measures and vector measures on $\Om$ respectively. Let $\rho(\cdot, t)_{0\le t\le 1}\in \mathcal{W}_p(X)$ be a absolutely continuous curve connecting $\rho_0$ and $\rho_1$. Then according to \cite[Theorem 5.14]{santambrogio2015optimal}, there exists a field $\boldsymbol{w}: [0, 1]\to (\M)^d$ such that $\boldsymbol{w}(\cdot, t) \ll \rho(\cdot, t)$ (hence $\boldsymbol{w} = \rho \boldsymbol{v}$ for some vector field $\boldsymbol{v}$) and the following continuity equation holds:
\begin{equation}
    \partial_t \rho(\x, t) + \nabla \cdot \boldsymbol{w}(\x, t) = 0~ \text{on}~ \Om \times [0,1].
\end{equation}
Correspondingly, the $W_p$ distance can be recovered by solving the following problem (see \cite{santambrogio2015optimal} for more details):
\begin{multline}\label{dynamical W_p}
 W_p^p(\rho_0, \rho_1)=
 \inf_{\rho, \boldsymbol{w}} \Bigg\{\int_0^1 \left(\int_{\Om} \frac{\lvert \boldsymbol{w}(\x, t) \rvert^p}{\rho^{p-1}(\x, t)} \dif \x\right) \dif t: \\
 \partial_t \rho(\x, t) + \nabla \cdot \boldsymbol{w}(\x, t) = 0 ~ (\x, t)\in \Om \times (0,1),
~\rho(\x, 0) = \rho_0(\x), \rho(\x,1) = \rho_1(\x) \Bigg\}.
\end{multline}
The Beckmann formulation of $W_1$ can also be derived from this dynamical formulation. In fact, by considering $\boldm(\x)=\int_0^1 \boldsymbol{w}(\x, t) \dif t $, one can obtain \eqref{eq:BeckmannOT}.

To take into account of the mass change, several UOT problems have been proposed and they are connected in various ways \cite{barrett2009partial,caffarelli2010free,piccoli2014generalized,figalli2010optimal}. In particular, the Wasserstein-Fisher-Rao (or Kantorovich-Helliger) distance has been proposed in \cite{chizat2018unbalanced,chizat2018interpolating,kondratyev2016new,liero2018optimal} by added a source into the dynamics.
The corresponding static formulation as an extension of the classical Kantorovich is derived for UOT in \cite{chizat2018interpolating,chizat2018unbalanced, liero2018optimal}, using either so-called semi-couplings  \cite{chizat2018interpolating,chizat2018unbalanced} or the relaxation of the marginal constraints \cite{liero2018optimal}. 

In this paper, we will focus on the generalization of the Wasserstein-1 distance given in \eqref{eq:BeckmannOT}. In particular, we focus on the following Beckmann formulation of an unbalanced OT problem:
\begin{equation}\label{eq:BeckmannUOT}
W_1^\alpha(\rho_0, \rho_1):= \min \left\{\int_\Om \lvert \boldm \rvert + \alpha \lvert \eta \rvert \dif x :~ \boldm: \Om \to\R^d,~ \nabla \cdot \boldm=\rho_0 - \rho_1 + \eta  \right\},
\end{equation}
with suitable boundary conditions. More details can be seen in Section \ref{Problem descriptions}. One way to understand this is through the dynamic formulation of the UOT studied in \cite{chizat2018interpolating}, which is a generalization of the Benamou-Brenier formulation \eqref{dynamical W_p}.
The dynamic formulation is given by
\begin{equation}\label{WFR}
\begin{aligned}
    & \min_{\rho,\boldsymbol{w},\zeta} \int_{0}^{1}\left[\frac{1}{p} \int_{\Om} \frac{\lvert \boldsymbol{w}(\x, t)\rvert^p}{\rho^{p-1}(\x, t)} \mathrm{~d} \x + \alpha^{p} \frac{1}{q} \int_{\Omega} \frac{\lvert\zeta(\x, t)\rvert^q}{\rho^{q-1}(\x, t)} \mathrm{~d} \x\right] \mathrm{d} t,\\
    & \st~ \partial_t \rho(\x, t) + \nabla \cdot \boldsymbol{w}(\x, t) = \zeta(\x, t)~ \text{on}~ \Om \times [0,1],\\
    & \text{with}~ \rho(\x, 0) = \rho_0(\x), \rho(\x,1) = \rho_1(\x),
\end{aligned}
\end{equation}
where $p,q \geq 1$ and $\alpha > 0$ is the weight parameter of the source term. 
The functional in \eqref{WFR} penalizes the transportation with $p$-norm and the source change with $q$-norm respectively. When $p=q\in [1,\infty)$, this dynamic formulation gives a distance. Taking $p =q = 1$, and similarly letting $\boldm(\x) = \int_0^1 \boldsymbol{w}(\x, t) \dif t$ and $\eta(\x) = \int_0^1 \zeta(\x, t) \dif t$, the corresponding Beckmann formulation \eqref{eq:BeckmannUOT} can then be derived. One may refer to Lemma \ref{prop:uot1} for more details. Clearly, in this UOT problem, $\rho_0$ and $\rho_1$ do not necessarily to have the same mass  and the parameter $\alpha$ in problem \eqref{WFR} controls the penalization of the source term. As a last comment, one often requires $\Omega$ to be convex for the dynamic formulation \eqref{dynamical W_p} to give the Wasserstein distances when $p>1$. As can be seen, the Beckmann formulation \eqref{eq:BeckmannOT} for $p=1$ does not require the convexity of $\Omega$ and it is equivalent to \eqref{dynamical W_p}. This means that the convexity of $\Omega$ is not required to study the Beckmann formulation and the dynamical formulation for $p=1$. Analogously, we do not require the convexity of $\Omega$ in \eqref{eq:BeckmannUOT}.

Our main focus in this paper is the connection between the Beckmann formulation for UOT \eqref{eq:BeckmannUOT} and the Beckmann formulation for OT \eqref{eq:BeckmannOT} when $\rho_0$ and $\rho_1$ are probability measures with the same mass, particularly when they are solved numerically using some optimization algorithms. Specifically, we aim to study whether the numerical solution of the UOT one can somehow converge to that for the corresponding OT problem under suitable optimization algorithms.
Such a problem is closely related to the so-called $\Gm$-convergence and some related results have been investigated in literature already. In particular,  \cite{chizat2018unbalanced} gives the corresponding result for some static problems of the unbalanced OT via the $\Gm$-convergence, while \cite{gangbo2019unnormalized} mentions  some numerical evidence of the convergence for the Wasserstein-Fisher-Rao distance. We focus on the Beckmann formulation because it corresponds to the Earth mover distance and has been widely applied in data science \cite{li2018parallel, doi:10.1137/18M1219813,ennaji2020beckmann} and more importantly, it is easier for computation and more suitable for optimization algorithms.

Our contribution can be summarized as follows. First, we establish the $\Gm$-convergence between the Beckmann problem \eqref{eq:BeckmannUOT} and \eqref{eq:BeckmannOT}. Then in discrete settings, we provide an estimate of lower bound of the parameter $\alpha$ for the solution of UOT being the same as the OT problem not just the convergence of the optimal solution. Lastly, the discrete UOT problem can be solved by a primal-dual hybrid gradient method (a.k.a Chambolle-Pock algorithm) \cite{EZC10,chambolle2011first} and we also give the corresponding condition of the parameter $\alpha$ for the reduction of the iterates for UOT to that for OT.

The rest of the paper is organized as follows. First in Section \ref{sc:background}, we provide the definitions of the usual $\Gm$-convergence and the sequence $\Gm$-convergence, and the relationship between them. Particularly in \ref{subsc:useful results}, we summarize some useful lemmas and theorems for $\Gm$-convergence which will be applied in later demonstrations. Then in Section \ref{sc:Convergence from UOT to OT}, we derive the equivalence between the Beckmann formulations to the dynamical ones for both UOT and OT at the beginning, and prove the existence of minimizers of these two problems. After that we establish the $\Gm$-convergence between the UOT and OT problem. Later in Section \ref{sc:discrete}, the finite convergence in discrete problems and the asymptotic preserving property are built and also, we present the iterates of a primal-dual hybrid gradient method for both UOT and OT problems and show the similar convergence between them. At last, in Section \ref{sc:numerical experiments} some numerical experiments on shape deformation and partial color transfer are implemented to validate the theoretical results and the algorithm.  

\section{Background on $\Gamma$-convergence}\label{sc:background}

To investigate the convergence of the optimization problems and their optimizers, one often makes use of the theory of $\Gamma$-convergence \cite{braides2002gamma}. Here we first recall the definitions of the usual $\Gm$-convergence:
\begin{definition}\label{def:Gmconv}
    Let $(f_n)$ be a sequence of functionals on $X$. Define
    \begin{gather}
    \begin{split}
    & \Gamma\hbox{-}\limsup_{n\to\infty} f_n(x)=\sup_{N_x}\limsup_{n\to\infty}\inf_{y\in N_x}f_n(y),\\
    & \Gamma\hbox{-}\liminf_{n\to\infty} f_n(x)=\sup_{N_x}\liminf_{n\to\infty}\inf_{y\in N_x}f_n(y),
    \end{split}
    \end{gather}
    where $N_x$ ranges over all the neighborhoods of $x$. If there exists a functional $f$ defined on $X$ such that
    \begin{gather}
    \Gamma\hbox{-}\limsup_{n\to\infty} f_n=\Gamma\hbox{-}\liminf_{n\to\infty} f_n =f,
    \end{gather}
    then we say the sequence $(f_n)$ $\Gm$-converges to $f$.
\end{definition}
    
The benefit of $\Gm$-convergence is that any cluster point of the minimizers of a $\Gm$-convergent sequence $(f_n)$ is a minimizer of the corresponding $\Gm$-limit functional $f$. This result can be found in many references like \cite{braides2002gamma} and one may also refer to Lemma \ref{lemma1} later.

The verification of $\Gm$-convergence using Definition \ref{def:Gmconv} of the optimization problems in this paper is not that straightforward. Instead, we will make use of the results of $\Gm_{\seq}$-convergence studied in \cite{buttazzo1982gamma} to get some sufficient conditions for $\Gm$-convergence in Definition \ref{def:Gmconv} on product spaces and we will utilize them in our problems.

\subsection{Notations and definitions}   

We first introduce some definition and notations of $\Gm_{\seq}$-convergence in \cite{buttazzo1982gamma}.  Define the operators $\mathcal{L}(\cdot)$ and $\mathcal{G}(\cdot)$ as 
\begin{equation}
    \mathcal{L}(\epsilon) = 
    \left\{
        \begin{aligned}
            & \sup ~~& \epsilon = +1, \\
            & \inf ~~& \epsilon = -1,
        \end{aligned}
    \right.
    \qquad
    \mathcal{G}(\epsilon) = 
    \left\{
        \begin{aligned}
            & \limsup ~~ & \epsilon = +1, \\
            & \liminf ~~ & \epsilon = -1.
        \end{aligned}
    \right.
\end{equation}
Let $(f_n)$ be a sequence of functions defined on a topological space $X$ and
\begin{gather}
\cS(x_0):=\{\{x^n\}\subset X: x^n\to x_0\}
\end{gather}
be the set of sequences that converge to $x_0$. Define the $\Gm_{\seq}$-limits of $(f_n)$ at point $x_0$ as
\begin{equation}
    \Gm_{\seq}(\N^{\epsilon_0}, X^{\epsilon_1})\lim_n f_n(x_0) = \underset{\{x^n\}\in\cS(x_0)}{\mathcal{L}(\epsilon_1)} \underset{n}{\mathcal{G}(\epsilon_0)} f_n(x^n),
\end{equation}
where $\epsilon_i\in \{+1, -1\}$, $i=0,1$. 
    
The relation between $\Gm_{\seq}$-convergence and the usual $\Gm$-convergence is given as follows, and we include a short proof in Appendix \ref{apdx} for our presentation to be self-contained.
\begin{proposition}\label{connection}
    It holds that
    \begin{gather}
    \begin{split}
    & \Gm_{\seq}(\N^+, X^-)\lim_n f_n =\Gamma\hbox{-}\limsup_{n\to\infty} f_n ,\\
    & \Gm_{\seq}(\N^-, X^-)\lim_n f_n = \Gamma\hbox{-}\liminf_{n\to\infty} f_n.
    \end{split}
    \end{gather}
    Consequently, if $f:=\Gm_{\seq}(\N, X^-)\underset{n}{\lim}f_n$ exists, then 
    $(f_n)$ $\Gm$-converges to $f$.
\end{proposition}

Many functionals in practice are defined on some natural product space. 
For two topological spaces $X$ and $Y$ and $(f_n)$ defined on the product space $X \times Y$, we can similarly define the $\Gm_{\seq}$-limits of $(f_n)$ at point $(x_0, y_0)\in X\times Y$ for $\epsilon_i\in \{+1, -1\}, i=0,1,2$ as
\begin{equation}
    \Gm_{\seq}(\N^{\epsilon_0}, X^{\epsilon_1}, Y^{\epsilon_2}) \lim_n f_n(x_0, y_0) = \underset{\{x^n\}\in\cS(x_0)}{\mathcal{L}(\epsilon_1)} ~\underset{\{y^n\}\in\cS(y_0)}{\mathcal{L}(\epsilon_2)} ~ \underset{n}{\mathcal{G}(\epsilon_0)} f_n(x^n, y^n).
\end{equation}

Here we take the space $X \times Y$ as an example to clarify the notations. Suppose $\epsilon_0 = +1$, $\epsilon_1 = -1$, $\epsilon_2 = -1$, then we have
\begin{equation}
    \Gm_{\seq}(\N^{+}, X^{-}, Y^{-}) \lim_n f_n (x_0, y_0) = \inf_{\{x^n\}\in\cS(x_0)} \inf_{\{y^n\}\in\cS(y_0)} \limsup_n f_n(x^n, y^n),
\end{equation}
where for any given convergent sequence $\{x^n\}\in \cS(x_0)$ and $\{y^n\}\in \cS(y_0)$, the $\limsup$ operator (or $\liminf$) is taken over the functional value sequence $(f_n(x^n, y^n))$ and the $\inf$ (or $\sup$) operator is taken over all the sequence $\{x^n\}\in \cS(x_0)$ and $\{y^n\}\in \cS(y_0)$ converging to $x_0$ and $y_0$ respectively. Moreover, if the $\Gm_{\seq}$-limit is independent of the value of $\epsilon$, then we omit the sign in the $\Gm_{\seq}$-limit, i.e. if
\begin{equation}
    \Gm_{\seq}(\N^{+}, X^{-}, Y^{-}) \lim_n f_n (x_0, y_0) = \Gm_{\seq}(\N^{-}, X^{-}, Y^{-}) \lim_n f_n (x_0, y_0),
\end{equation}
we can write $\Gm_{\seq}(\N, X^{-}, Y^{-}) \lim_n f_n (x_0, y_0)$ for simplicity. The notations are similar for the spaces $X$ and $Y$.

\subsection{Useful results}\label{subsc:useful results}

In this subsection, we summarize several useful lemmas and theorems for $\Gm_{\seq}$-convergence from \cite{buttazzo1982gamma}. In particular, these results provide tools to check  $\Gm$-convergence on product spaces. For the completeness,  we give simplified   proofs of the lemmas and theorem appeared in this section in Appendix \ref{apdx}.

The following lemma  states that any cluster point of the minimizers of a $\Gm$-convergent sequence is the minimizer of the corresponding $\Gm$-limit functional.
\begin{lemma}\label{lemma1}
	Let $X$ be a topological space, and let $(f_n)$ be a sequence of functionals mapping from $X$ to $\bar{\R} = [-\infty, +\infty]$. If
	\begin{equation*}
	    \Gm_{\seq}(\N, X^-)\lim_n f_n = f,
	\end{equation*}
	then 
	\begin{equation}
	    \inf_X f \geq \limsup_n[\inf_X f_n].
	\end{equation}
	Moreover, if there exists a sequence $(x^n)$ converging to some $x_0\in X$, with
	\begin{equation*}
	    \liminf_n f_n(x^n) = \liminf_n [\inf_X f_n],
	\end{equation*}
	then
	\begin{equation}
	    f(x_0) = \inf_X f = \lim_n[\inf_X f_n].
	\end{equation}
\end{lemma}

The UOT problem in consideration is naturally defined on a product space and the functional is of the form $J+\one_E$. We introduce some related results in \cite{buttazzo1982gamma} in this regard.  

\begin{lemma}\label{lemma2}
	Let $X, Y$ be two topological spaces and $(f_n), (g_n)$ be two sequences of functionals defined on the product space $X \times Y$ to $\bar{\R}^{+} = [0, +\infty]$, and let $(x_0, y_0) \in X \times Y$. Suppose there exists $a, b \in \bar{\R}^{+}$ such that
	\begin{equation*}
	\begin{split}
	&\Gm_{\seq} (\N, X^-, Y) \lim_n f_n(x_0, y_0) = a,\\
	&\Gm_{\seq} (\N, X, Y^-) \lim_n g_n(x_0, y_0) = b.
	\end{split}
	\end{equation*}
	Then it holds that
	\begin{equation*}
	    \Gamma_{\seq} (\N, X^-, Y^-)\lim_n (f_n + g_n) (x_0,y_0) = a + b.
	\end{equation*}
\end{lemma}

Suppose $X$ is a topological space and $E$ is a set in $X$, and the indicator function of $E$ is defined as follows:
    \begin{equation}
        \one_E(x) = \left\{
        \begin{aligned}
                & 0, \quad &\text{if $x \in E$},\\
                & +\infty, \quad & \text{otherwise}.
        \end{aligned}
        \right.
    \end{equation}
The following lemma gives a sufficient condition of a sequence of the set indicator functions to be $\Gm_{\seq}$-convergence.

\begin{lemma}\label{lemma3}
    Suppose $\{E_n\}$ is a sequence of sets in space  $X \times Y$. If there exists a set $E_\infty \subset X \times Y$ satisfies the following two conditions:
    \begin{itemize}
        \item If $x^n \rightarrow x$, $y^n \rightarrow y$ and $(x^n, y^n) \in E_n$ for infinitely many $n$, then $(x, y) \in E_\infty$;
        \item If $(x,y) \in E_\infty$ and $x^n \rightarrow x$, then there exists $y^n \rightarrow y$ such that $(x^n, y^n) \in E_n$ for $n$ large enough,
    \end{itemize}
    then $\one_{E_\infty} = \Gm_{\seq}(\N, X, Y^-) \underset{n}{\lim} \one_{E_n}$.
\end{lemma}

The following theorem provides us the criterion to check $\Gm$-convergence of the functional of the form $J+\one_E$ on product spaces.
\begin{theorem}\label{main}
	Suppose $X$ and $Y$ are two topological spaces. $(J_n)$ is a sequence of functionals defined on the product space $X \times Y$ and $(E_n)$ is a sequence of set in $X \times Y$.
Suppose that $J_n$ and $\one_{E_n}$ are sequential $\Gm$-convergent in the following sense
\begin{equation*}
\begin{split}
& J_\infty = \Gamma_{\seq}(\N, X^-, Y)\lim_n J_n,\\
& \one_{E_\infty} = \Gamma_{\seq}(\N, X, Y^-)\lim_n \one_{E_n},
\end{split}
\end{equation*}
then $J_n + \one_{E_n}$ is $\Gm$-convergent to $J_\infty + \one_{E_\infty}$ in the sense of  Definition \ref{def:Gmconv}.

Consequently, for every $n \in \N^+$, let $(x_n, y_n)$ be an optimal pair of the optimization problem
\[
\min_{X \times Y} (J_n + \one_{E_n}).
\]
If $x_n \rightarrow x_\infty$ in $X$ and $y_n \rightarrow y_\infty$ in $Y$, then $(x_\infty, y_\infty)$ is an optimal pair of the problem
\[
\min_{X \times Y} (J_\infty + \one_{E_\infty}).
\]
\end{theorem}


\section{Convergence from UOT to OT}\label{sc:Convergence from UOT to OT}

In this section, we establish the convergence of the Beckmann formulation of the UOT problem \eqref{eq:BeckmannUOT}  to the corresponding OT problem \eqref{eq:BeckmannOT} in the sense of $\Gamma$-convergence.

\subsection{Problem descriptions}\label{Problem descriptions}

Fix a bounded domain $\Om \subset \R^d$ with smooth boundary which is not necessarily convex. Suppose that $\rho_0$ and $\rho_1$ are two probability measures defined on $\Om$.
To obtain the full description of the mathematical problems, one needs to specify the no-flux boundary condition $\boldm\cdot \boldsymbol{n} = 0$ on $\partial\Omega$ by the physical significance. Hence, the Beckmann formulation of the UOT problem is given by:
\begin{equation}\label{eq:uotmath0}
\begin{aligned}
   \min_{ \boldm, \eta} \quad & \int_\Omega \lvert \b{m}(\x) \rvert  + \alpha \lvert \eta(\x) \rvert \dif \x,\\
   \st \quad & \nabla \cdot \boldm + \rho_1 -\rho_0 = \eta ~ \text{in $\Om$}, \\
   & \boldm \cdot \boldsymbol{n} = 0 ~ \text{on $\partial \Om$}.
\end{aligned}
\end{equation}
Correspondingly, the Beckmann formulation of the traditional OT problem is analogously given by:
\begin{equation}\label{eq:otmath0}
\begin{aligned}
   \min_{ \boldm} \quad & \int_\Omega \lvert \boldm(\x) \lvert \dif \x,\\
   \st \quad & \nabla \cdot \boldm + \rho_1 -\rho_0 = 0 ~ \text{in $\Om$}, \\
   & \boldm \cdot \b{n} = 0 ~ \text{on $\partial \Om$}.
\end{aligned}
\end{equation}
The constraint in \eqref{eq:uotmath0} is understood in the weak sense, i.e.
\begin{equation}
    -\int_\Omega \boldm \cdot \nabla \varphi ~ \dif \x  +\int_\Omega (\rho_1 - \rho_0 - \eta)\varphi~ \dif \x = 0, \quad \text{for} ~ \forall \varphi \in C_b^1(\bar{\Omega}).
\end{equation}
The constraint for \eqref{eq:otmath0} is understood similarly.

Before we start the analysis, let us clarify its connection to the dynamic formulation as announced in the introduction. 
Recall the UOT problem for the case $p = q = 1$ in \eqref{WFR}:
\begin{equation}
\label{eq:uot}
\begin{aligned}
    & \min_{\rho, \boldsymbol{w}, \zeta}  \int_\Omega \int_0^1 \lvert \boldsymbol{w}(\x, t) \rvert + \alpha \lvert \zeta(\x, t) \rvert \dif t \dif \x, \\
    & \st \quad  \partial_t \rho(\x, t) + \nabla \cdot \boldsymbol{w}(\x, t) = \zeta(\x, t)~ \text{in $\Omega \times [0, 1]$}, \\
    & \operatorname{with}~  \boldsymbol{w}(\x, t)\cdot\boldsymbol{n} = 0, \text{on $\partial\Omega \times [0, 1]$}\\
    &\quad\quad \rho(\x, 0) = \rho_0(\x), \rho(\x,1) = \rho_1(\x).
\end{aligned}
\end{equation}
Here, $w(\cdot, t) \in (\mathcal{M}(\Om))^d, 0\le t\le 1$ is a $d$-dimensional vector field  and $\zeta(\cdot, t) \in \mathcal{M}(\Om), 0\le t\le 1$ is a source term on $\Om$. Note that $\mathcal{M}(\Om)$ is the set of Radon measures, the dual space of $C_b(\Om)$. Also, $t\mapsto \rho(\cdot, t)\in  \mathcal{M}(\Om)$ is a path on  $\mathcal{M}(\Om)$. 
Define 
\begin{equation}\label{m&eta}
    \boldm(\x) = \int_0^1 \boldsymbol{w}(\x, t)\dif t, \quad \eta(\x) = \int_0^1 \zeta(\x, t)\dif t,
\end{equation}
then we have the following lemma:

\begin{lemma}\label{prop:uot1}
    Under the settings above, the Beckmann formulation of UOT \eqref{eq:uotmath0} is equivalent to the dynamical UOT problem \eqref{eq:uot}. The same conclusion for the OT case.
\end{lemma}
\begin{proof}
 On the one hand, for any feasible pair $(w, \zeta)$ in \eqref{eq:uot}, it holds that
    \begin{equation}
    \begin{aligned}
        \int_\Om \lvert \boldm(\x) \rvert + \alpha \lvert \eta(\x) \rvert \dif \x & = \int_\Om \bigg\lvert \int_0^1 w(\x, t) \dif t \bigg\rvert  + \alpha\bigg\lvert \int_0^1  \zeta(\x, t) \dif t \bigg\rvert \dif \x \\
        & \leq \int_\Om\int_0^1 \lvert w(\x, t) \rvert + \alpha \lvert \zeta(\x, t) \rvert\dif t \dif \x,
    \end{aligned}
    \end{equation}
    therefore one can obtain that
    \begin{equation}
        \min_{\boldm, \eta} \int_\Om \lvert \boldm(\x) \rvert + \alpha \lvert \eta(\x) \rvert \dif \x \leq \min_{w, \zeta} \int_\Om\int_0^1 \lvert w(\x, t) \rvert + \alpha \lvert \zeta(\x, t) \rvert\dif t \dif \x. 
    \end{equation}
    On the other hand, for any $\alpha > 0$ suppose $(\boldm_\alpha, \eta_\alpha)$ is an optimal pair to problem \eqref{eq:uotmath0}. Then let $w_\alpha(\x, t) \equiv \boldm_\alpha(\x)$ and $\zeta_\alpha(\x, t) \equiv\eta_\alpha(\x)$ for $\forall t \in [0, 1]$ and define
    \begin{equation*}
        \rho_\alpha(\x, t) = t\rho_1(\x) + (1 - t)\rho_0(\x),
    \end{equation*}
    one can get that the triad $(\rho_\alpha,w_\alpha,\zeta_\alpha)$ is a feasible solution to problem \eqref{eq:uot}. Therefore, 
    \begin{equation}
    \begin{aligned}
        \min_{\rho ,\boldsymbol{w}, \zeta}  \int_\Omega \int_0^1 \lvert \boldsymbol{w}(\x, t) \rvert + \alpha \lvert \zeta(\x, t) \rvert \dif t \dif \x & \leq \int_\Omega \int_0^1 \lvert \boldsymbol{w_\alpha}(\x, t) \rvert + \alpha \lvert \zeta_\alpha(\x, t) \rvert \dif t \dif \x\\
        & = \min_{ \boldm, \eta} \quad \int_\Omega \lvert \b{m}(\x) \rvert  + \alpha \lvert \eta(\x) \rvert \dif \x.
    \end{aligned}
    \end{equation}
    Hence, combining the two part one can conclude that the problem \eqref{eq:uot} and \eqref{eq:uotmath0} are equivalent.
\end{proof}

In the previous formulation, the form $\lvert \boldm(x) \rvert$ is the Euclidean norm ($\ell_2$-norm) of $ \boldm$ and from now on, we consider the general case of $\ell_p$-norm for $\boldm(x)$ where $1 \leq p < +\infty$. In other words, we use
\begin{gather*}
|\boldm(\x)|_p:=(\sum_{i=1}^d |m_i(\x)|^p)^{1/p}, ~~\text{for } p \in [1, +\infty),
\end{gather*}
to replace the original $\ell_2$-norm $\lvert \boldm(x) \rvert$. Moreover, for the convenience in later analysis, taking $\xi := \alpha \eta$ in problem \eqref{eq:uotmath0} and adding the term  $\xi$ to the objective function in problem \eqref{eq:otmath0} as a free variable, we obtain the two equivalent UOT and OT problems respectively:
\begin{equation}\label{eq:uotmath}
\begin{aligned}
   \min_{\boldm, \xi} \quad & \int_\Omega \lvert \boldm(\x) \lvert_p  + \lvert \xi(\x) \lvert \dif \x,\\
   \st \quad & \nabla \cdot \boldm + \rho_1 -\rho_0 = \frac{1}{\alpha} \xi ~ \text{in $\Om$}, \\
   & \boldm \cdot \boldsymbol{n} = 0 ~ \text{on $\partial \Om$}.
\end{aligned}
\end{equation}
and
\begin{equation}\label{eq:otmath}
\begin{aligned}
   \min_{\boldm, \xi} \quad & \int_\Omega \lvert \boldm(\x) \lvert_p + \lvert \xi(\x) \rvert \dif \x,\\
  \st \quad & \nabla \cdot \boldm + \rho_1 -\rho_0 = 0 ~ \text{in $\Om$}, \\
   & \boldm \cdot \boldsymbol{n} = 0 ~ \text{on $\partial \Om$}.
\end{aligned} 
\end{equation}

Then our goal comes to build a connection between the UOT problem \eqref{eq:uotmath} and the OT problem \eqref{eq:otmath}, by establishing the convergence from problem \eqref{eq:uotmath} to \eqref{eq:otmath} as $\alpha \rightarrow +\infty$ in the sense of $\Gamma$-convergence. 

\subsection{Existence of minimizers}
In this subsection, we first show the existence of minimizers of the above UOT problems. For the convenience of the discussion, we define the total variation norm of the fields $\boldm$ and $\xi$ as follows:
\begin{equation}\label{TV norm}
     \|\boldm\|:=\int_{\Omega}|\boldm(\x)|_p\dif\x, \qquad \| \xi \| := \int_\Om | \xi(\x) | \dif \x.
\end{equation}
 It is well known that the total variation norm is in fact the dual norm against the bounded continuous functions since $\Omega$ is bounded.

We first note that the continuity equation constraint  depends only on the gradient  in $\boldm$. By Helmholtz decomposition, we obtain
\begin{gather}
\boldm=-\nabla\Phi+\boldh,\quad \nabla\cdot \boldh=0,\quad
 \frac{\partial \Phi}{\partial \boldsymbol{n}}=0, \quad\boldh\cdot \boldsymbol{n} =0,
\end{gather}
where $\Phi$ is a scalar field and $\boldh$ is a field without divergence and the constraint is imposed on $\Phi$:
\begin{equation}\label{eq:measureconstraint}
    \int_\Omega \nabla\Phi \cdot \nabla \varphi ~ \dif \x  +\int_\Omega \left(\rho_1 - \rho_0 - \inv{\alpha} \xi\right)\varphi~ \dif \x = 0, \quad \text{for} ~ \forall \varphi \in C_b^1(\bar{\Omega}).
\end{equation}
In our setting, the divergence free condition $\nabla\cdot \boldh=0$ should be understood in the weak sense. Hence, we introduce the following space
\begin{equation*}
\mathcal{H}:=\left\{\boldh\in \M^d:  \int_{\Omega}\boldh\cdot \nabla\varphi\,\dif\x=0, \forall 
\varphi\in C_b^1(\bar{\Omega}) \right\}.
\end{equation*}

We note that the Helmholtz decomposition always exists if $\xi$ has a bounded total variation by the lemma below, from \cite[Theorem 22, Lemma 23]{brezis1973semi}.
\begin{lemma}\label{lmm:Poisson}
For each $\xi\in \M$ with $\int |\xi|\,dx<\infty$, there exists a weak solution $\Phi  \in W^{1,q}(\Om)$ where $q\in [1, d/(d-1))$ with $\int \Phi\,dx=0$ to the Poisson equation
\begin{equation}\label{eq:Poisson}
    -\Delta u = \xi.
\end{equation}
The solution satisfies
\begin{equation}\label{eq:control}
    \| \Phi  \|_{W^{1,q}} \leq C_q\| \xi \|
\end{equation}
for some constant $C_q > 0$ depending on $\alpha$  and $\Omega$ only.
\end{lemma}
Note that the result in \cite{brezis1973semi} is for $\xi \in L^1(\Omega)$ while we are considering Radon measures here, although there is no essential difference. By the inequality \eqref{eq:control}, it yields that the TV norm of $\nabla\Phi$ defined in \eqref{TV norm} can be controlled as $\|\nabla \Phi\|\le C\|\xi\|$.  Then, for each $(\xi, \boldm)$ satisfying the constraint \eqref{eq:measureconstraint}, by Lemma \ref{lmm:Poisson} one can find a weak solution $\Phi$ with $\int \Phi\,\dif\x=0$ satisfying
\begin{equation*}
    \Delta \Phi = \rho_1 - \rho_0 - \frac{1}{\alpha}\xi.
\end{equation*}
Define
\begin{equation*}
    \boldh:=\boldm+\nabla\Phi\in \mathcal{H}.
\end{equation*}
and one can get that the Helmholtz decomposition of $\boldm$ exists and is stable.

Hence, the problem (\ref{eq:uotmath}) is reduced to
\begin{gather}\label{eq:uotreform}
\begin{aligned}
\min_{(\xi, \boldh, \Phi)} & \int_{\Omega}|\boldh-\nabla\Phi|_p+|\xi(\x)|\,\dif\x,\\
\st & \int_\Omega \nabla\Phi \cdot \nabla \varphi ~ \dif \x  +\int_\Omega \left(\rho_1 - \rho_0 - \inv{\alpha} \xi\right)\varphi~ \dif \x = 0, \quad \text{for} ~ \forall \varphi \in C_b^1(\bar{\Omega}), \\ 
 & \boldh\in \mathcal{H}.
\end{aligned}
\end{gather}
Similarly, the OT problem \eqref{eq:otmath} becomes
\begin{gather}\label{eq:otreform}
\begin{aligned}
\min_{(\xi, \boldh, \Phi)} & \int_{\Omega}|\boldh-\nabla\Phi|_p+|\xi(\x)|\,\dif\x,\\
\st & \int_\Omega \nabla\Phi \cdot \nabla \varphi ~ \dif \x  +\int_\Omega (\rho_1 - \rho_0)\varphi~ \dif \x = 0, \quad \text{for} ~ \forall \varphi \in C_b^1(\bar{\Omega}), \\ 
 & \boldh\in \mathcal{H}.
\end{aligned}
\end{gather}
Using these two reduced problems, we can establish the following existence results.

\begin{proposition}\label{pro:exist}
Both \eqref{eq:uotmath} and \eqref{eq:otmath} have global minimizers
over $\M\times (\M)^d$.
\end{proposition}
\begin{proof}
By the reformulation above, we prove the existence results for \eqref{eq:uotreform} and \eqref{eq:otreform}. We will take \eqref{eq:uotreform} as an example.

First of all, we equip the set $\M\times\mathcal{H}$ for $(\xi, \boldh)$ with the weak topology: $(\xi_n, \boldh_n) \Rightarrow (\xi, \boldh)$ if
\begin{equation}\label{weak convergence}
    \int f d\xi_n+\int \boldsymbol{g}\cdot d\boldh_n \to \int f d\xi+\int \boldsymbol{g}\cdot d\boldh, \forall f\in C_b^1(\bar{\Omega}; \R), \boldsymbol{g}\in C_b^1(\bar{\Omega}; \R^d).
\end{equation}
Clearly, the space $\mathcal{H}$ is closed in $\M^d$ under the weak topology. Consider the functional:
\begin{equation}
    (\xi, \boldh)\mapsto F(\xi, \boldh):=\int_{\Omega}|\boldh(\x)-\nabla\Phi(\x;\xi)|_p+|\xi(\x)|\,\dif\x = \|\boldh-\nabla\Phi(\cdot;\xi)\| + \|\xi \|,
\end{equation}
where $\Phi(\cdot; \xi)$ indicates that $\Phi$ is solved according to the Poisson equation with given $\xi$ .

It is  straightforward to verify that $F$ is lower semi-continuous under the topology for $\M\times \mathcal{H}$. In fact, if 
$(\xi_n, \boldh_n)\Rightarrow (\xi, \boldh)$, one has that
\[
\boldh_n-\nabla \Phi(\cdot; \xi_n)\Rightarrow \boldh-\nabla\Phi(\cdot; \xi).
\]
To see this, for any test vector field $\boldsymbol{g} \in C_b^1(\bar{\Omega}; \R^d)$, one can also decompose $\boldsymbol{g}$ as
\begin{equation}
    \boldsymbol{g} = \nabla\phi+\boldsymbol{v}, \quad \nabla\cdot \boldsymbol{v}=0, \quad \boldsymbol{v}\cdot \boldsymbol{n} = 0 \text{ on }\partial \Om.
\end{equation}
Then, 
\begin{multline}
\int_{\Omega}\nabla \Phi(\cdot; \xi_n)\cdot \boldsymbol{g}\,\dif\x
=\int_{\Omega}\nabla \Phi(\cdot; \xi_n)\cdot \nabla\phi \,\dif\x
=\int_{\Omega}\left(\frac{1}{\alpha}\xi_n+\rho_0-\rho_1\right)\phi\,\dif\x\\
\longrightarrow \int_{\Omega}\left(\frac{1}{\alpha}\xi+\rho_0-\rho_1\right)\phi\,\dif\x
=\int_{\Omega}\nabla \Phi(\cdot; \xi)\cdot \boldsymbol{g}\,\dif\x.
\end{multline}
Since bounded smooth functions are dense in the space of bounded continuous functions under the topology of uniform convergence (recall that $\bar{\Omega}$ is a bounded set), the above therefore holds for all $\boldsymbol{g} \in C_b^1(\bar{\Omega}; \R^d)$.

Consequently,
\begin{equation}
    \|\boldh-\nabla\Phi(\cdot; \xi)\|+\|\xi\|
\le \liminf_{n\to\infty }\|\boldh_n-\nabla\Phi(\cdot; \xi_n)\|+\|\xi_n\|.
\end{equation}
Hence, the lower semicontinuity is established.

It is clear that
\begin{equation*}
   F_*:=\inf_{(\xi, \boldh)\in \M\times \mathcal{H}}F(\xi, \boldh)>-\infty.
\end{equation*}
Then, consider a minimizing sequence, $(\xi_n, \boldh_n)$
such that $F(\xi_n, \boldh_n)\to F_*$. Then for this minimizing sequence, one has
\begin{equation*}
    \sup_n \|\boldh_n-\nabla\Phi(\cdot; \xi_n)\|+\|\xi_n\|<+\infty. 
\end{equation*}
According to \eqref{eq:control}, $\|\nabla\Phi(\cdot; \xi_n)\|$
is also uniformly bounded. Consequently, 
\begin{equation}
    \sup_n \|\boldh_n\|+\|\xi_n\|\leq \sup_n \|\boldh_n-\nabla\Phi(\cdot; \xi_n)\| + \|\nabla\Phi(\cdot; \xi_n) \|+\|\xi_n\| < +\infty.
\end{equation}
The Banach-Alaoglu theorem indicates that there must be a weakly convergent subsequence. Hence, together with the lower semi-continuity, the minimizer exists. 
\end{proof}

\subsection{Convergence}
By noticing the conditions in Theorem \ref{main}, we will regard $\xi$
and $\Phi$ as independent variables. Define the functional $J$ for all
 $(\xi, \boldh, \Phi) \in (\M\times \mathcal{H})\times W^{1,1}(\Omega)$ by
\begin{equation}
    J(\xi, \boldh, \Phi) = \int_\Omega \lvert \boldh (\x)-\nabla\Phi(\x) \rvert_p + \lvert \xi(\x) \rvert \dif \x.
\end{equation}
We equip  the space for $(\xi, \boldh)$ with the weak convergence of the measures defined in \eqref{weak convergence}
\begin{gather*}
X:=\left(\M \times \mathcal{H}, ~\Rightarrow\right).
\end{gather*}
As mentioned above already, the space $\mathcal{H}$ is closed in $\M^d$ under the weak topology.

Note that this weak topology for measures is closer to the weak* convergence in functional analysis.  Moreover, the topology we choose for the space of $\nabla \Phi$ is the total variation norm, or the $W^{1,1}$ norm of $\Phi$ (assuming $\Phi$ has mean zero)
\begin{gather*}
Y:=W^{1,1}(\Omega).
\end{gather*}

Now, we introduce the set of constraints
\begin{multline}
E_\alpha :=\Bigg\{((\xi, \boldh), \Phi) \in X\times Y:
\int_{\Omega}\Phi\,\dif\x=0, \\
\int_\Omega \nabla\Phi \cdot \nabla \varphi ~ \dif \x  +\int_\Omega (\rho_1 - \rho_0 - \inv{\alpha} \xi)\varphi~ \dif \x = 0, 
\quad  \forall \varphi \in C_b^1(\bar{\Omega})
\Bigg\}.
\end{multline}
Similarly, 
\begin{multline}
E_\infty :=\Bigg\{((\xi, \boldh), \Phi) \in X\times Y:
\int_{\Omega}\Phi\,\dif\x=0, \\
\int_\Omega \nabla\Phi \cdot \nabla \varphi ~ \dif \x  +\int_\Omega (\rho_1 - \rho_0)\varphi~ \dif \x = 0, 
\quad  \forall \varphi \in C_b^1(\bar{\Omega})
\Bigg\}.
\end{multline}
Problem \eqref{eq:uotmath} can be reformulated as 
\begin{equation}
\min_{(\xi, \boldh)\in X,\ \Phi\in Y}~ J(\xi, \boldh, \Phi) + \one_{E_\alpha}.
\end{equation}
Similarly,  \eqref{eq:otmath} is
\begin{equation}
\min_{(\xi, \boldh)\in X, \ \Phi\in Y}~ J(\xi, \boldh, \Phi) + \one_{E_\infty}.
\end{equation}
The following theorem states the convergence from (\ref{eq:uotmath}) to (\ref{eq:otmath}) as $\alpha$ goes to infinity.
 \begin{theorem}\label{thm:convuot}
    Suppose for any $\alpha > 0$, $(\boldm^\alpha, \xi^\alpha)$ is an optimal solution of the corresponding UOT problem \eqref{eq:uotmath}. 
Then
\begin{enumerate}[(i)]
\item  With a decomposition
\[
\boldm^{\alpha} = \boldh^{\alpha} - \nabla\Phi^{\alpha}
\] 
such that $((\xi^\alpha, \boldh^\alpha), \Phi^\alpha)\in \one_{E_\alpha}$. There exists some constant $\alpha_0>0$, $M>0$ such that for all $\alpha\ge \alpha_0$
\[
\sup_\alpha\|\boldh^\alpha\|+\|\nabla\Phi^\alpha\|+\|\xi^\alpha\|\le M.
\]

\item 
For any increasing sequence $\{\alpha_I\}$ going to infinity, where $I$ is an index set, there exists a convergent subsequence $((\xi^{\alpha_k}, \boldh^{\alpha_k}), \Phi^{\alpha_k})\in X\times Y$ with $\alpha_k \uparrow +\infty$ such that the limit $((\xi^{\infty}, \boldh^{\infty}), \Phi^{\infty})\in \one_{E_{\infty}}$ and $(\xi^{\infty}, \boldm^{\infty})=(\xi^{\infty}, \boldh^{\infty}-\nabla\Phi^{\infty})$ is a solution of the OT problem \eqref{eq:otmath}. Moreover, $\xi_\infty = 0$. 
\end{enumerate}
\end{theorem}

To prove this theorem, we first show the $\Gamma$-convergence of $J+\one_{E_\alpha}$
to $J+\one_{E_\infty}$.
\begin{lemma}\label{lmm:Gamconv}
With the above setup, $J+\one_{E_\alpha}$ is $\Gm$-convergent to $J+\one_{E_\infty}$.
\end{lemma}
\begin{proof}
Here, we verify the two conditions in Theorem \ref{main}. 

We will first show that $\Gm_{\seq}(\N, X^-, Y)\underset{\alpha}{\lim} J(\xi^\alpha, \boldh^\alpha, \Phi^\alpha) = J(\xi, \boldh, \Phi)$. It suffices to prove the following two results:
\begin{equation*}
\begin{aligned}
     \inf_{(\xi^\alpha, \boldh^\alpha) \Rightarrow (\xi, \boldh)} \sup_{\Phi^\alpha \rightarrow \Phi} \limsup_{\alpha\ra+\infty} J(\xi^\alpha, \boldh^\alpha, \Phi^\alpha) &\leq J(\xi, \boldh, \Phi), \\
     \inf_{(\xi^\alpha, \boldh^\alpha) \hra (\xi, \boldh)}\inf_{\Phi^\alpha \rightarrow \Phi} \liminf_{\alpha\ra+\infty} J(\xi^\alpha, \boldh^\alpha, \Phi^\alpha) &\geq J(\xi, \boldh, \Phi).
\end{aligned}
\end{equation*}
These two relations, by Sandwich theorem, can ensure that both both the signs of $\N$ and $Y$ in the $\Gm_{\seq}$-limit can be omitted.

For any pair $((\xi, \boldh), \Phi)$ and for any convergent sequence $\Phi^\alpha \to \Phi$ in $W^{1,1}$, one can choose a particular weak convergent sequence $\left\{(\xi^\alpha, \boldh^\alpha) \right\}$ such that $(\xi^\alpha, \boldh^\alpha) \hra (\xi, \boldh)$, $\| \xi^\alpha \| \rightarrow \| \xi \|$ and that $\|\boldh^\alpha-\nabla\Phi^\alpha\|\to \|\boldh-\nabla\Phi \|$. Such sequence of $(\xi^\alpha, \boldh^\alpha)$ clearly exists (for example, one can choose the constant sequence $(\xi^\alpha, \boldh^\alpha)=(\xi, \boldh)$). Then, one has
\begin{equation*}
\begin{aligned}
    \limsup_{\alpha\ra+\infty} J(\xi^\alpha, \boldh^\alpha, \Phi^\alpha) & = \limsup_{\alpha\ra+\infty}\left( \int_\Om \lvert \xi^\alpha(\x)\rvert \dif \x + \int_\Om \lvert \boldh^\alpha(\x)-\nabla\Phi^\alpha(\x) \rvert_p \dif \x \right) \\
    & = \| \xi \| + \| \boldh-\nabla\Phi \| = J(\xi, \boldh, \Phi).
\end{aligned}
\end{equation*}
Hence, it holds that
\begin{equation}
    \inf_{(\xi^\alpha, \boldh^\alpha) \Rightarrow (\xi, \boldh)} \sup_{\Phi^\alpha \to \Phi} \limsup_{\alpha\ra+\infty} J(\xi^\alpha, \boldh^\alpha, \Phi^\alpha) \leq J(\xi, \boldh, \Phi).
\end{equation}
\begin{remark}
The strong convergence of $\Phi^\alpha$ here is essential to obtain the limit
$J(\xi, \boldh, \Phi)$ as an upper bound. If there is only weak convergence of $\nabla\Phi$ as used in the proof of Proposition \ref{pro:exist}, such an upper bound can not be established. 
\end{remark}

On the other hand, for any weak convergent sequence $(\xi^\alpha, \boldh^\alpha)\hra (\xi, \boldh)$ in $X$ and $\Phi^\alpha \to \Phi$ in $Y$, one has $\|\xi\|\le \liminf \|\xi^\alpha\|$
and $\|\boldh-\nabla\Phi\|\le \lim\|\boldh^\alpha-\nabla\Phi^\alpha\|$. Consequently,
\begin{equation*}
\begin{aligned}
    \liminf_{\alpha\ra+\infty}J(\xi^\alpha, \boldh^\alpha, \Phi^\alpha) & = \liminf_{\alpha\ra+\infty}\left( \int_\Om \lvert \xi^\alpha(\x)\rvert \dif \x + \int_\Om \lvert \boldh^\alpha(\x)-\nabla\Phi^\alpha(\x) \rvert_p \dif \x \right) \\
    & \geq \| \xi \| + \| \boldh-\nabla\Phi \| = J(\xi, \boldh, \Phi).
\end{aligned}
\end{equation*}
It follows that
\begin{equation}
    \inf_{(\xi^\alpha, \boldh^\alpha) \hra (\xi, \boldh)}\inf_{\Phi^\alpha \ra \Phi} \liminf_{\alpha\ra+\infty} J(\xi^\alpha, \boldh^\alpha, \Phi^\alpha) \geq J(\xi, \boldh, \Phi).
\end{equation}
Combining the two formulas, one obtains
\begin{equation}
    \Gm_{\seq}(\N,X^-, Y)\underset{\alpha}{\lim} J(\xi^\alpha, \boldh^\alpha, \Phi^\alpha) = J(\xi, \boldh, \Phi).
\end{equation}

Next, we will show $\one_{E_\infty} =\Gamma_{\seq}(\N, X, Y^-)\underset{\alpha}{\lim}\one_{E_\alpha}$. 
Using  Lemma \ref{lemma3}, it suffices to show that
\begin{itemize}
    \item [(i)] If $(\xi^\alpha, \boldh^\alpha) \Rightarrow (\xi, \boldh)$, $\Phi^\alpha \ra \Phi$ in $W^{1,1}$, $((\xi^\alpha, \boldh^\alpha), \Phi^\alpha) \in E_\alpha$ for infinitely many $\alpha$, then $((\xi, \boldh), \Phi) \in E_\infty$;
    \item [(ii)] If $((\xi, \boldh), \Phi) \in E_\infty$ and $(\xi^\alpha, \boldh^\alpha) \Rightarrow (\xi, \boldh)$, then there exists $\Phi^\alpha \ra \Phi$ such that $((\xi^\alpha, \boldh^\alpha), \Phi^\alpha) \in E_\alpha$ for $\alpha$ large enough.
\end{itemize}
\noindent
For (i), we consider the sequence $\alpha$ such that $((\xi^\alpha, \boldh^\alpha), \Phi^\alpha) \in E_\alpha$, then
\begin{equation}\label{3.17}
\int_{\Omega}\Phi^\alpha\,\dif\x=0, \quad \int_\Omega \nabla\Phi^\alpha \cdot \nabla \varphi ~ \dif \x  +\int_\Omega (\rho_1 - \rho_0 - \inv{\alpha} \xi^\alpha)\varphi~ \dif \x = 0, \quad \forall \varphi \in C_b^1(\bar{\Omega}).
\end{equation}
Since $\varphi\in C_b$ and $\nabla\varphi\in C_b$, one clearly has 
\[
\int_\Omega \boldm^\alpha \cdot \nabla \varphi
\to \int_\Omega \boldm \cdot \nabla \varphi, \quad 0=\int_{\Omega}\Phi^\alpha\,\dif\x\to \int_{\Omega}\Phi\,\dif\x.
\]
As $\xi^\alpha \Rightarrow \xi$, it is uniformly bounded and 
\begin{equation}
    \lim_{\alpha\ra+\infty} \int_\Om \frac{1}{\alpha} \xi^\alpha \varphi ~ \dif \x = 0.
\end{equation}
Hence, it is  easy to see that $(\xi, \boldh, \Phi) \in E_\infty$. Note that here $\lim \int_{\Omega}\boldh^\alpha\cdot\nabla\varphi\,\dif\x=\int_{\Omega}\boldh\cdot\nabla\varphi\,\dif\x$ using the fact that  $X$ is closed under the weak convergence of measures.

\noindent For (ii) we consider the following Poisson equation:
\begin{equation}\label{Laplacian}
    \left\{
        \begin{aligned}
             -\Delta u & = \frac{1}{\alpha} \xi^\alpha ~ \text{in}~ \Om, \\
             \frac{\partial u}{\partial \boldsymbol{n}} & = 0 ~ \text{on} ~\partial\Om.
        \end{aligned}
    \right.
\end{equation}
 Here the sequence $\left\{\xi^\alpha\right\}$ is given in (ii) which weakly converges to $\xi$.  Consequently, by Lemma \ref{lmm:Poisson}, there exists
$\phi^\alpha$ with $\int \phi^\alpha\,\dif\x=0$ and
\begin{equation*}
    \lim_{\alpha\ra+\infty} \| -\nabla \phi^\alpha \| \leq \lim_{\alpha\ra+\infty} \frac{C}{\alpha} \| \xi^\alpha \| = 0.
\end{equation*}

Define $\Phi^\alpha = \Phi +\phi^\alpha$, one clearly has $\int \Phi^\alpha\,\dif\x=0$,
$\Phi^\alpha\to \Phi$
and by the definition of the weak solution of the Poisson equation that
\begin{equation}
    \int_\Omega \nabla\Phi^\alpha \cdot \nabla \varphi ~ \dif \x  +\int_\Omega \left(\rho_1 - \rho_0 - \inv{\alpha} \xi^\alpha \right)\varphi~ \dif \x = 0, \quad \text{for} ~ \forall \varphi \in C_b^1(\bar{\Omega}),
\end{equation}
which implies that $(\xi^\alpha, \boldh^\alpha, \Phi^\alpha) \in E_\alpha$ for all $\alpha$.
\end{proof}

Now, we prove the main result in this section.
\begin{proof}[Proof of Theorem \ref{thm:convuot}]

Suppose $\boldm$ is a feasible solution to problem \eqref{eq:otmath}.
Clearly, $(0, \boldm)$ is also a feasible solution to problem \eqref{eq:uotmath} for any $\alpha > 0$. Therefore,
\begin{equation*}
    \int_\Omega \lvert \boldm^\alpha \rvert_p + \lvert \xi^\alpha \rvert \dif \x \leq \int_\Omega \lvert \boldm \rvert_p \dif \x < +\infty.
\end{equation*}
where $(\xi^\alpha, \boldm^\alpha)$ is an optimal solution of problem \eqref{eq:otmath}. 
Then, by Lemma \ref{lmm:Poisson}, there exists $\Phi^\alpha\in W^{1,1}$ with $\int \Phi^\alpha\,\dif\x=0$ that is a weak solution to
\[
-\Delta\Phi^\alpha+\rho_1-\rho_0=\frac{1}{\alpha}\xi^\alpha, \quad \frac{\partial\Phi^\alpha}{\partial n}=0
\]
with
\[
  \|\Phi^\alpha\|_{W^{1,1}}\le C\|\rho_0-\rho_1-\frac{1}{\alpha}\xi^\alpha\|
  \le C\left(2+\frac{1}{\alpha}\int|\boldm|_p\,\dif\x\right).
\]
Moreover, define
\[
\boldh^\alpha=\boldm^\alpha+\nabla\Phi^\alpha.
\]
It is easy to see that $\boldh^\alpha\in \mathcal{H}$ and consequently, $((\xi^\alpha, \boldh^\alpha), \Phi^\alpha)\in \one_{E_\alpha}$. Moreover,
\[
\|\boldh^\alpha\|\le \|\boldm^\alpha\|+\|\nabla\Phi^\alpha\|\le 
C\left(2+(1+\alpha^{-1})\int|\boldm|_p\,\dif\x \right).
\]
The first claim follows if $\alpha\ge \alpha_0>0$.

We now show that for any optimal sequence $\{(\xi^\alpha, \boldm^\alpha)\}$ with 
$\boldm^\alpha=\boldh^\alpha-\nabla\Phi^\alpha$ as above, there exists a convergent subsequence $(\xi^{\alpha^k}, \boldh^{\alpha^k}) \Rightarrow (\xi, \boldh)$ and $\Phi^{\alpha_k} \to \Phi$.

Using the Banach-Alaoglu theorem we have that any bounded set in $X$ is precompact. Consequently, there is a subsequence  $(\xi^{\alpha_k}, \boldh^{\alpha_k}) \Rightarrow (\xi, \boldh)\in X$.
Moreover, let $\Phi$ with $\int\Phi\,\dif\x=0$ be the solution to
\begin{equation}
    \left\{\begin{aligned}
        & -\Delta \Phi + \rho_1 - \rho_0 =0 ~ \text{in}~ \Om,\\
        & \frac{\partial \Phi}{\partial \boldsymbol{n}} = 0 ~ \text{on} ~ \partial \Om.
    \end{aligned}\right.
\end{equation}
Since $\xi^{\alpha_k}\Rightarrow \xi$, $\xi^{\alpha_k}$ is thus uniformly bounded. Then as $\alpha_k \rightarrow +\infty$, one has
\[
\|\Phi^{\alpha_k}-\Phi\|_{W^{1,1}}\le \frac{C}{\alpha_k}\|\xi^{\alpha_k}\|  \to 0.
\]

Using Lemma \ref{lmm:Gamconv}, one obtains  $\Gm$-convergence of the functional. Using Theorem \ref{main}  it follows that $(\xi, \boldh, \Phi)$ is a minimizer of $J+\one_{E_\infty}$. Hence, the conclusion follows.  Moreover, since $(\xi^\infty, \boldm^\infty)$ is an optimal pair of problem \eqref{eq:otmath}, it is obvious that  $\xi^\infty$ should be $0$. 
\end{proof}

\section{Convergence in discrete setting and the asymptotic preserving property}\label{sc:discrete}
In this section, we focus on the discretized problems of the Beckmann formulation for UOT and OT problems. We take $\Om$ to be a bounded rectangular domain in $\R^d$. We will show that the convergence from UOT to OT is preserved in the discrete setting so that the discretization is asymptotic preserving, which guarantees that a numerical method for the optimization problems can be applied for the discrete problems in a uniform manner, and the optimizer of the UOT can converge to that for the OT problem along the limit. Moreover, we show that when $\alpha$ is larger than some critical value that is only related to the dimension $d$ and the width of the domain $L$, and independent to the mesh size $h$, the minimizer of the discrete UOT problem is then reduced to that for OT, which we call finite convergence. We also present the algorithm proposed in \cite{li2018parallel} and applied to solve both UOT and OT problem and show that the iterates for UOT will reduce to that for OT as the penalty parameter $\alpha > M$ for some constant $M$ dependant on the discrete problem merely.

\subsection{Discretized problems}
We first formulate the discrete UOT and OT problems. We use the same discrete scheme as \cite{li2018parallel, doi:10.1137/18M1219813}. Let $\Om = [0, 1]^d$ be a $d$ dimensional rectangular domain, and $\Om_h$ be the discrete mesh-grid of $\Om$ with step size $h$, i.e.
\begin{equation}\label{discrete Omega}
    \Om_h = \{0, h, 2h, \cdots, 1\}^d.
\end{equation}
Let $N = 1/h$ be the grid size. For a more general case $\Om = [0, L]^d$, it can be transformed to $\Om = [0, 1]^d$ by scaling. For all $x \in \Omega_h$, $x$ is a $d$-dimensional vector, where the  $i$-th component $x_i$ takes values from $\{0,h,2h,3h,\cdots,1\}$. The discretized distributions $\rho_h^0 = \{\rho^0(x)\}_{x \in \Omega_h} , \rho_h^1 = \{\rho^1(x)\}_{x \in \Omega_h}$ and $\eta_h = \{\eta(x)\}_{x\in \Om_h}$ are all $(N + 1)^d$ tensors. The discretized flux $\boldm_h = \{ \boldm(x) \}_{x \in \Omega_h}$ is a $(N + 1)^d\times d$ tensor, which can be regarded as a map from $\Omega_h $ to $\mathbb{R}^d$. Then the discretized problem for  \eqref{eq:uotmath0} is 
\begin{equation}\label{duot}
\begin{aligned}
\underset{m_h,\eta_h}{\min}
&\sum_{x \in \Omega_h} \left( \lvert \boldm_h(x) \rvert_p h^d  + \alpha\lvert \eta_h(x) \rvert h^d\right)  \\
\st ~~ & \mathrm{div}^h(\boldm_h(x)) - \eta_h(x) = \rho_h^0(x) - \rho_h^1(x),\qquad \forall x \in \Omega_h, \\
\end{aligned}	
\end{equation}
where the discrete boundary conditions are given such that $\boldm_{h,i}(x_{-i}, x_i) = 0$ if $x_i = 1$ and $\boldm_{h,i}(x_{-i}, x_i - h) = 0$ if $x_i = 0$ for $\forall i \in \{1,2,\cdots, d\}$, and  $\sum_{x\in\Om_h}\eta_h(x) = 0$. Here the notion “-$i$” refers to all the components excluding $i$, i.e.
\begin{equation*}
   x_{-i} = (x_1, \cdots, x_{i - 1}, x_{i + 1}, \cdots, x_d), 
\end{equation*}
and for any $z \in \R$
\begin{equation*}
    \boldm_{h,i}(x_{-i}, z) = \boldm_{h,i}(x_1, \cdots, x_{i - 1}, z, x_{i + 1}, \cdots, x_d).
\end{equation*}
Note that we use the ghost point $(x_{-i}, 0 - h) = 0$ for each $i$ and therefore, only the condition for $x_i=1$ is used explicitly in the domain $\Om_h$ while the condition at $x_i=0$ is only implicitly used for the definition of the discrete divergence operator $\mathrm{div}^h(\cdot)$, which is defined as
\begin{equation*}
\mathrm{div}^h(\boldm_h(x)) = \sum_{i = 1}^{d}D_{h,i}\boldm_h(x), \\ 
\end{equation*}
and for $\forall i \in \{1,2,\cdots, d\}$
\begin{equation}
D_{h,i}\boldm_h(x) = \left\{
\begin{aligned}
& (\boldm_{h,i}(x_{-i},x_i))/h, \qquad x_i = 0, \\
& (\boldm_{h,i}(x_{-i},x_i) - \boldm_{h,i}(x_{-i},x_i - h))/h, \qquad 0 < x_i < 1, \\
& (-\boldm_{h,i}(x_{-i},x_i - h))/h, \qquad x_i = 1, \\
\end{aligned}	
\right.
\end{equation}
which makes the discrete approximation be consistent with the zero-flux boundary condition. In the above definition, $\boldm_h(x) \in \mathbb{R}^d$ denotes  the flow at point $x$ and $ \boldm_{h,i}(x) \in \mathbb{R}$ denotes the $i$-th component of $\boldm_{h}(x)$. 
Moreover, we define $f(\cdot) = \sum_{\Om_h}\lvert \cdot \rvert_p h^d$ as a discrete $\ell_{p,1}$ norm on $\Om_h$, then the problem \eqref{duot} can be reformulated as 
\begin{equation}\label{duot1}
\begin{aligned}
\underset{\boldm_h,\eta_h}{\min}~ & f(\boldm_h) + \alpha f(\eta_h), \\
\st ~~ & \mathrm{div}^h (\boldm_h) - \eta_h = \rho_h.
\end{aligned}	
\end{equation}
 Similarly, the discrete OT problem (\ref{eq:otmath0}) is given as
\begin{equation}\label{dot}
\begin{aligned}
\underset{\boldm_h}{\min}~ & f(\boldm_h)\\
\st ~~ & \ddiv (\boldm_h) = \rho_h,
\end{aligned}	
\end{equation}
with the zero-flux boundary condition in \eqref{duot}.

\subsection{A primal-dual hybrid algorithm}
With the discrete formulation, we can apply a primal-dual hybrid algorithm \cite{EZC10,chambolle2011first} to solve both the UOT and OT problems. Note  that this algorithm is also adopted in  \cite{li2018parallel} for the OT problem. We first give some definitions on the discrete space $\Om_h$:
\begin{equation*}
    \langle \boldm_h, \boldm_h^{\prime} \rangle_h = \sum_{x\in\Om_h}\boldm_h(x)\boldm_h^{\prime}(x) h^d,  \quad \lVert \boldm_h \rVert_{h,2}^2 = \sum_{x\in \Om_h}\lvert \boldm(x) \lvert_2^2 h^d.
\end{equation*} 
 For the OT problem (\ref{dot}), we solve the following min-max reformulation
\begin{equation}
   \min_{\boldm_h}\max_{\varphi_h} ~L(\boldm_h,\varphi_h) =
f(\boldm_h) + \langle \varphi_h, \ddiv \boldm_h - \rho_h \rangle_{h}
\end{equation}
by the primal-dual hybrid algorithm whose updating rule is given as
\begin{equation}
\begin{aligned}
& \boldm_h^{k + 1} = \underset{\boldm_h}{\arg\min}~ L(\boldm_h,\varphi_h^k) + \frac{1}{2\mu}\lVert \boldm_h - \boldm_h^k \rVert_{h,2}^2, \\
& \tilde{\boldm}_h^{k + 1} = 2\boldm_h^{k + 1} - \boldm_h^k, \\
& \varphi_h^{k + 1} = \underset{\varphi_h}{\arg\max}~L(\tilde{\boldm}_h^{k + 1},\varphi_h^k) - \frac{1}{2\tau}\lVert \varphi_h - \varphi_h^k \rVert_{h,2}^2.
\end{aligned}	
\end{equation}
The update is equivalent to
\begin{equation}\label{CP_ot}
\begin{aligned}
& \boldm_h^{k + 1} = \mathrm{Prox}_{\mu f}(\boldm_h^k - \mu (\ddiv)^*(\varphi_h^k)), \\
& \tilde{\boldm}_h^{k + 1} = 2\boldm_h^{k + 1} - \boldm_h^k, \\
& \varphi_h^{k + 1} = \varphi_h^{k} + \tau(\ddiv(\tilde{\boldm}_h^{k+ 1}) - \rho_h),
\end{aligned}	
\end{equation}
where $\operatorname{Prox}_{\mu f}$ is the proximity operator of function $\mu f$ defined as
\begin{equation*}
    \mathrm{Prox}_{\mu f}(\boldm_h^k - \mu (\ddiv)^*(\varphi_h^k)) = \underset{\boldm_h}{\arg\min}~ \mu f(\boldm_h) + \frac{1}{2}\lVert \boldm_h - (\boldm_h^k - \mu (\ddiv)^*(\varphi_h^k)) \rVert_{h,2}^2,
\end{equation*}
and $\mu$, $\tau$ are  algorithmic parameters and $(\ddiv)^*$ represents the conjugate operator of $\ddiv$. 

By the definition of conjugate operator, for $\forall \boldm_h$ and $\forall u_h \in \Om_h$, we have
\begin{equation*}
	\langle \ddiv (\boldm_h), u_h\rangle_h = \langle \boldm_h, (\ddiv)^* (u_h)\rangle_h,
\end{equation*}
then it is easy to check that $(\ddiv)^* = - \nabla_h$ and 
\begin{equation*}
	\nabla_h(u_h) = (\partial_{h,1} u_h, \partial_{h,2} u_h, \cdots, \partial_{h,d} u),
\end{equation*}
where each $\partial_{h,i}u_h$ is 
\begin{equation}
	\partial_{h,i} u_h(x_1, x_2, \cdots, x_d) = \left\{
		\begin{aligned}
			& (u(x_{-i}, x_i + h) - u(x_{-i}, x_i)) / h, \quad 0 \leq x_i < 1, \\
			& 0, \quad x_i = 1
		\end{aligned}
	\right.
\end{equation}
for $i = 1,2,\cdots, d$. According to  \cite{chambolle2011first}, the algorithm is ensured to be convergent if $\mu \tau \|\ddiv\|^2 < 1$.  

Similarly, we solve the UOT problem (\ref{duot}) with the following min-max reformulation
\begin{equation}
\underset{\boldm_h,\eta_h}{\min}\underset{\varphi_h}{\max} ~L(\boldm_h,\eta_h,\varphi_h) =
f(\boldm_h) + \alpha f(\eta_h) + \langle \varphi_h,\ddiv(\boldm_h) - \eta_h - \rho_h \rangle_{h},	
\end{equation}
and by the primal-dual hybrid algorithm as follows
\begin{equation}
\begin{aligned}
& (\boldm_h^{k + 1},\eta_h^{k + 1}) = \underset{\boldm_h,\eta_h}{\arg\min~} L(\boldm_h,\eta_h,\varphi_h^k) + \frac{1}{2\mu}(\lVert \boldm_h - \boldm_h^k \rVert_{h,2}^2 + \lVert \eta_h - \eta_h^k \rVert_{h,2}^2), \\
& \tilde{\boldm}_h^{k + 1} = 2\boldm_h^{k + 1} - \boldm_h^k, 
\qquad \tilde{\eta}_h^{k + 1} = 2\eta_h^{k + 1} - \eta_h^k, \\
& \varphi_h^{k  + 1} = \underset{\varphi_h}{\arg\max}~L(\tilde{\boldm}_h^{k + 1},\tilde{\eta}_h^{k + 1},\varphi_h^k) - \frac{1}{2\tau}\lVert \varphi_h - \varphi_h^k \rVert_{h,2}^2, 
\end{aligned}
\end{equation}
which can be written as
\begin{equation}\label{CP_uot}
\begin{aligned}
& \boldm_h^{k + 1} = \mathrm{Prox}_{\mu f}(\boldm_h^k - \mu (\ddiv)^*(\varphi_h^k)) \\
& \eta_h^{k + 1} = \mathrm{Prox}_{\alpha\mu f}(\eta_h^k + \mu \varphi_h^k) \\
& \tilde{\boldm}_h^{k + 1} = 2\boldm_h^{k + 1} - \boldm_h^k,
\qquad \tilde{\eta}_h^{k + 1} = 2\eta_h^{k + 1} - \eta_h^k, \\
& \varphi_h^{k + 1} = \varphi_h^{k} + \tau(\ddiv(\tilde{\boldm}_h^{k+ 1}) - \tilde{\eta}_h^{k+ 1} - \rho_h).
\end{aligned}	
\end{equation}
The algorithm is convergent if $\mu\tau \lVert [\ddiv, -I] \rVert^2 < 1$ and we terminate the algorithm when the primal-dual gap $R_n^k$ 
\begin{equation}\label{termat}
\begin{aligned}
R_h^k & = \frac{1}{\mu}(\lVert \boldm_h^{k + 1} - \boldm_h^{k} \rVert_{h,2}^2 + \lVert \eta_h^{k+1} - \eta_k^h \rVert^2_{h,2})
+ \frac{1}{\tau}\lVert \varphi_h^{k + 1} - \varphi_h^{k} \rVert_{h,2}^2 \\
& - 2\langle \varphi_h^{k + 1} - \varphi_h^{k},\ddiv(\boldm_h^{k + 1} - \boldm_h^{k}) - (\eta_h^{k + 1} - \eta_h^{k})\rangle_h ,
\end{aligned}
\end{equation}
falls below a predefined threshold $\epsilon$.

\subsection{Convergence of the discrete UOT problem}

In Section \ref{sc:Convergence from UOT to OT} we showed the $\Gm$-convergence from UOT to OT in continuous case. For the discrete problem, we also have the similar proposition.
\begin{proposition}
    Define $X_h := \mathbb{R}^{(N + 1)^d}$ and $Y_h := \mathbb{R}^{(N + 1)^d\times d}$. Taking $\xi_h := \alpha \eta_h$ in (\ref{eq:uotmath0}) we have that
    \begin{equation}
        J_h(\xi_h, \boldm_h) := \sum_{x \in \Omega_h} \left( \lvert \boldm_h(x) \rvert_p h^d  + \lvert \xi_h(x) \rvert h^d\right),
    \end{equation}
    and we define
    \begin{equation*}
    \begin{aligned}
        & E_h^\alpha := \left\{(\xi_h, \boldm_h) \in X_h \times Y_h: \ddiv(\boldm_h) - \frac{1}{\alpha} \xi_h - \rho_h = 0. \right\},\\
        & E_h^\infty := \left\{(\xi_h, \boldm_h) \in X_h \times Y_h: \ddiv(\boldm_h) - \rho_h = 0. \right\}.
    \end{aligned}
    \end{equation*}
    Then  we have the $\Gm$-convergence from $J_h + \one_{E_h^\alpha}$ to $J_h + \one_{E_h^\infty}$.
\end{proposition}
\begin{proof}
    For discrete case, the convergence in the space $X_h$ and $Y_h$ reduces to pointwise convergence, and it is obvious that
    \begin{equation}
        J_h(\xi_h, \boldm_h) = \Gm_{\seq}(\mathbb{N}, X_h^{-}, Y_h)\lim_\alpha J_h(\xi_h, \boldm_h).
    \end{equation}
    Then our goal is to verify that $E_h^\alpha$ and $E_h^\infty$ also satisfy the conditions given in Lemma \ref{lemma3}:
    \begin{itemize}
        \item [(i)] If $\xi_h^\alpha \rightarrow \xi_h, \boldm_h^\alpha \rightarrow \boldm_h$ and $(\xi_h^\alpha, \boldm_h^\alpha) \in E_h^\alpha$ for infinitely many $\alpha$, then $(\xi_h, \boldm_h) \in E_h^\infty$;
        \item [(ii)] If $(\xi_h, \boldm_h) \in E_h^\infty$ and $\xi_h^\alpha \rightarrow \xi_h$, then there exists $\boldm_h^\alpha \rightarrow \boldm_h$ such that $(\xi_h^\alpha, \boldm_h^\alpha) \in E_h^\alpha$ for $\alpha$ large enough.
    \end{itemize}
    For (i), due to the fact that $(\xi_h^\alpha, \boldm_h^\alpha) \in E_h^\alpha$ for infinitely many $\alpha$, we have
    \begin{equation}\label{ehalpha}
        \ddiv(\boldm_h^\alpha) - \frac{1}{\alpha}\xi_h^\alpha - \rho_h = 0.
    \end{equation}
    As $\xi_h^\alpha \rightarrow \xi_h$, the sequence $\left\{\xi_h^\alpha\right\}$ is uniformly bounded and $\ddiv(\boldm_h^\alpha)$ is a simple matrix-vector multiplication for discrete problem, then let $\alpha$ goes to infinity in the equality \eqref{ehalpha} we  obtain that:
    \begin{equation}\label{lineareqs0}
        \ddiv(\boldm_h) - \rho_h = 0,
    \end{equation}
    which leads to $(\xi_h, \boldm_h) \in E_h^\infty$. \\
    
    For (ii), as $(\xi_h, \boldm_h) \in E_h^\infty$, for any $\alpha > 0$ and the given sequence $\left\{\xi_h^\alpha\right\}$, it suffices to solve the equations directly:
    \begin{equation}\label{lineareqs}
        \ddiv(\dm_h^\alpha) = \frac{1}{\alpha}\xi_h^\alpha,
    \end{equation}
    where $\dm_h^\alpha$ also satisfies the discrete boundary conditions $\dm_{h,i}(x_{-i}, x_i - h) = 0$ if $x_i = 1$ for $\forall i \in \{1,2,\cdots, d\}$. For any $u \in \operatorname{ker}((\ddiv)^*)$, $\nabla_h u = 0$ yields $u \equiv C$ for some constant $C$. Therefore, for $\forall u \in \operatorname{ker}((\ddiv)^*)$, we have
    \begin{equation*}
        \langle \xi_h^\alpha, u \rangle_h = C\sum_{x \in \Om_h}\xi_h^\alpha(x)h^d = 0,
    \end{equation*}
    which implies $\xi_h^\alpha \perp \operatorname{ker}((\ddiv)^*)$. Therefore the linear system \eqref{lineareqs} is soluble for any $\alpha > 0$. Moreover, with the condition $\xi_h^\alpha \rightarrow \xi_h$, it is obvious that the right side in \eqref{lineareqs} $\frac{1}{\alpha}\xi_h^\alpha \rightarrow 0$ as $\alpha$ goes to infinity. Hence for each $\alpha$, we choose the solution that has the least norm $\dm_h^\alpha$ (which should be perpendicular to kernel space of $\ddiv$) such  that $\dm_h^\alpha \rightarrow 0$ as $\alpha \rightarrow +\infty$. Then define 
    \begin{equation*}
        \boldm_h^\alpha = \boldm_h + \dm_h^\alpha
    \end{equation*}
    and correspondingly $\boldm_h^\alpha \rightarrow \boldm_h$ and $(\xi_h^\alpha, \boldm_h^\alpha) \in E_h^\alpha$ for any $\alpha > 0$.\\
    Using Lemma \ref{lemma3} we obtain that $\one_{E_h^\infty} = \Gm_{\seq}(\mathbb{N}, X_h, Y_h^{-})\underset{\alpha}{\lim}\one_{E_h^\alpha}$, and using Theorem \ref{main} one can conclude that $J_h + \one_{E_h^\alpha}$ is $\Gm$-convergent to $J_h + \one_{E_h^\infty}$.
\end{proof}

The $\Gm$-convergence for discrete problems indicate that the minimizers, which are bounded obviously, would have a convergent subsequence with the limit being the minimizer of the discrete OT problem.

In fact, for the discrete problem, we can show a stronger convergence for the minimizers of the two problems, as shown in the following theorem.
\begin{theorem}\label{discrete convergence}
	Suppose $\Om = [0, L]^d$ where $d$ is the dimension of the space and $L$ is the length of the interval in each dimension. Let $\Om_h= [0, L]^d_h$ be the discrete meshgrid of $\Om$ defined in (\ref{discrete Omega}) and $Nh = L$.  If $\alpha > \frac{dL}{2}$, then the optimal $\eta_h^*$ of the discrete UOT problem equals to $0$. Consequently, the minimizer reduces to that for OT.
\end{theorem}

\begin{proof}
    Different with before, here we will use the optimal conditions to show the result. Using the Lagrangian multiplier and omit the scaling $h^d$, the discrete UOT problem (\ref{duot1}) is equivalent to the following min-max problem:
	\begin{equation}\label{dminimax}
		\underset{\boldm_h,\eta_h}{\min} \underset{\varphi_h}{\max} ~~ \sum_{x \in \Om}\left(\lvert \boldm_h(x) \rvert_p + \alpha \lvert \eta_h(x) \rvert\right) + \langle \varphi_h, \ddiv (\boldm_h) - \eta_h - \rho_h \rangle,
	\end{equation}
  where the inner product $\langle \cdot, \cdot \rangle = \frac{1}{h^d}\langle \cdot, \cdot \rangle_h$.  Suppose $(\boldm_h^*, \eta_h^*,\varphi_h^*)$ is an optimal solution to (\ref{dminimax}), then by the first-order optimal conditions, for each point $x \in \Om_h$ we have that
\begin{equation}\label{first order condition}
	\left\{
	\begin{aligned}
		& 0 \in \partial \lvert \boldm_h^*(x) \rvert_p - \nabla_h(\varphi_h^*(x)), \\
		& 0 \in \alpha \partial \lvert(\eta_h^*(x) \rvert - \varphi_h^*(x), \\
		& 0 = \ddiv (\boldm_h^*(x)) - \eta_h^*(x) - \rho_h(x).
	\end{aligned}
	\right.
\end{equation}
Here  $\partial \lvert \boldm_h^*(x) \rvert_p$ and $\partial \lvert(\eta_h^*(x) \rvert$ represent the subgradients of  $\ell_p$-norm of vector $\boldm_h^*(x)$ and the absolute value of $\eta_h^*(x)$ respectively. From the second condition we obtain that for any point $(x_{i_1}, x_{i_2},\cdots, x_{i_d})\in (0,L]_h^d$,
\begin{equation*}
	-\alpha \leq \varphi_h^*(x_{i_1}, x_{i_2},\cdots, x_{i_d}) \leq \alpha.
\end{equation*}
To show $\eta_h^* = 0$, without loss of generality, we suppose that there exist $(x_{i_1}, x_{i_2},\cdots, x_{i_d})$ in $[0,L]^d_h$ such that $\eta_h^*(x_{i_1}, x_{i_2},\cdots, x_{i_d}) > 0$, then by the second condition we have 
\begin{equation*}
	\varphi_h^*(x_{i_1}, x_{i_2},\cdots, x_{i_d}) = \alpha.
\end{equation*}
Notice that for any $p \in [1, +\infty)$, the subgradient of $\partial \lvert \boldm_h^*(x) \rvert_p$ is defined as
\begin{equation*}
    \partial | \boldm_h^*(x) |_p = \left\{\textbf{v} \in \R^d: |\textbf{v}|_q \leq 1,~ \textbf{v}^\top \boldm_h^*(x) = | \boldm_h^*(x) |_p \right\},    
\end{equation*}
where $1 / p + 1 / q = 1$ and for $p = 1$, $\lvert \textbf{v} \rvert_q$ implies  $L_\infty$-norm of vector $\textbf{v}$. Then by the first condition in (\ref{first order condition}), we have that for each $x \in \Om_h$,
\begin{equation*}
    | \nabla \varphi_h^* (x) |_q \leq 1,
\end{equation*}
which indicates that the absolute value of each component of $\nabla \varphi_h^*(x)$ is also less than 1. More precisely, we have the following different situations:

$\bullet$ For $i_k \neq 0, 1$, at the point $(x_{-i_k}, x_{i_k})$ and $(x_{-i_k}, x_{i_k} - h)$ we can get that
\begin{equation}\left \{
	\begin{aligned}
		& \nabla_h (\varphi_h^*(x_{-i_k}, x_{i_k})) \in \partial \lvert(\boldm_{h}^*(x_{-i_k}, x_{i_k})\rvert_p, \\
		& \nabla_h (\varphi_h^*(x_{-i_k}, x_{i_k} - h)) \in \partial \lvert \boldm_{h}^*(x_{-i_k}, x_{i_k} - h) \rvert_p,
	\end{aligned}
\right.
\end{equation}
which leads to
\begin{equation}
	\left\{
	\begin{aligned}
		& \frac{\varphi_h^*(x_{-i_k}, x_{i_k} + h) - \varphi_h^*(x_{-i_k}, x_{i_k})}{h} \in \partial \lvert \boldm_{h,k}^*(x_{-i_k}, x_{i_k}) \rvert.\\
		& \frac{\varphi_h^*(x_{-i_k}, x_{i_k}) - \varphi_h^*(x_{-i_k}, x_{i_k} - h)}{h} \in \partial \lvert \boldm_{h,k}^*(x_{-i_k}, x_{i_k} - h) \rvert.
	\end{aligned}
	\right.
\end{equation}

$\bullet$ For $i_k = 1$, at the point $(x_{-i_k}, x_{i_k} - h)$ we can get that
\begin{equation*}
	\nabla_h (\varphi_h^*(x_{-i_k}, x_{i_k} - h)) \in \partial \lvert \boldm_{h}^*(x_{-i_k}, x_{i_k} - h) \rvert_p,
\end{equation*}
i.e. 
\begin{equation}
	\frac{\varphi_h^*(x_{-i_k}, x_{i_k}) - \varphi_h^*(x_{-i_k}, x_{i_k} - h)}{h} \in \partial \lvert \boldm_{h,k}^*(x_{-i_k}, x_{i_k} - h) \rvert.
\end{equation}
	
$\bullet$ For $i_k = 0$, at the point $(x_{-i_k}, x_{i_k})$ we can get that
\begin{equation*}
	\nabla_h(\varphi_h^*(x_{-i_k}, x_{i_k})) \in \partial \lvert \boldm_{h}^*(x_{-i_k}, x_{i_k}) \rvert_p,
\end{equation*}
i.e. 
\begin{equation}
	\frac{\varphi_h^*(x_{-i_k}, x_{i_k} + h) - \varphi_h^*(x_{-i_k}, x_{i_k})}{h} \in \partial \lvert\boldm_{h,k}^*(x_{-i_k}, x_{i_k}) \rvert.
\end{equation}
For all the cases we get that (if exists)
\begin{equation}
\left\{
	\begin{aligned}
	& -1 \leq \frac{\varphi_h^*(x_{-i_k}, x_{i_k}) - \varphi_h^*(x_{-i_k}, x_{i_k} - h)}{h} \leq 1, \quad k = 1,2,\cdots, d, \\
	& -1 \leq \frac{\varphi_h^*(x_{-i_k}, x_{i_k} + h) - \varphi_h^*(x_{-i_k}, x_{i_k})}{h} \leq 1, \quad k = 1,2,\cdots, d. \\
\end{aligned}
\right.
\end{equation}
Similarly,  we can continue to use the first condition at all the points $(x_{-i_k}, x_{i_k} - h)$ and $(x_{-i_k}, x_{i_k} + h)$ for $k = 1,2,\cdots, d$ (if exists) and get the restriction of $\varphi_h^*$ at any point on the meshgrid $[0,\, L]^d_h$:
\begin{equation}\label{second condition}
	\alpha - dL = \alpha - dNh \leq \varphi_h^*(x_{\tilde{i}_1}, x_{\tilde{i}_2},\cdots, x_{\tilde{i}_d}) \leq \alpha \quad 
	\text{for $\forall ~ (x_{\tilde{i}_1}, x_{\tilde{i}_2}, \cdots, x_{\tilde{i}_d}) \in [0,\, L]^d_h$}.
\end{equation}
On the other hand, as  $\rho_h^0$ and $\rho_h^1$  are equal mass, from the third condition we  get
\begin{equation*}
	\sum_{x \in \Omega_h} \eta_h^*(x) = \sum_{x \in \Omega_h}(A_h \boldm_h^*) - \sum_{x \in \Omega_h} \rho_h = 0. 
\end{equation*}
For $\eta_h^*(x_{i_1}, x_{i_2},\cdots, x_{i_d}) > 0$, there must exist another point $(x_{j_1}, x_{j_2}, \cdots, x_{j_d})$ in the space $\Om_h$ such that $\eta_h^*(x_{j_1}, x_{j_2}, \cdots, x_{j_d}) < 0$. Then by the second condition, we get $\varphi_h^*(x_{j_1}, x_{j_2}, \cdots, x_{j_d}) = -\alpha$. Then for $\alpha > dL / 2$, we have  
\begin{equation*}
	\varphi_h^*(x_{j_1}, x_{j_2}, \cdots, x_{j_d}) = -\alpha < \alpha - d L,
\end{equation*}
which is contradiction to the condition (\ref{second condition}). Therefore the optimal $\eta_h^* \equiv 0$ on the space $\Om_h$ and $\boldm_h^*$ is a solution to the OT problem. 
\end{proof}

\begin{remark}
    Note that the condition given in the theorem is only sufficient, in practice the exact threshold of $\eta_h^* \equiv 0$ tends to be smaller than $dL / 2$. Moreover though we only prove for $\Om = [0, L]^d$, the theorem holds for the general triangular case $\Om_h = [a_1, b_1] \times [a_2, b_2] \times \cdots [a_d, b_d]$ and the corresponding condition of $\alpha$ should be changed as $\alpha > d \underset{i}{\max}\{b_i - a_i\} / 2$.
\end{remark}

The above result indicate that the Beckmann formulation is advantageous as the usual discretization is asymptotic-preserving, which means that the convergence of UOT to OT can be preserved as $\alpha\to \infty$. This property is beneficial in the sense that when a numerical method is applied to the discretization problems, the numerical method can reduce to the one for OT as $\alpha\to\infty$ automatically.

In fact, as for the iterations \eqref{CP_ot} and \eqref{CP_uot} in the primal dual algorithm, with some restriction to the parameter $\alpha$, we can get the connection between them, which is stated as the following theorem:

\begin{theorem}
    Suppose in \eqref{CP_ot} and \eqref{CP_uot}, $\mu$ and $\tau$ are set the same which satisfy the convergence condition $\mu\tau \lVert [\ddiv, -I] \rVert^2 < 1$. Then there exists a constant $M > 0$ such that for all $\alpha > M$,  we have that $\eta_h^k \equiv 0$ on $\Om_h$ for $k$ large enough. Correspondingly, the subsequent iterates of UOT (\ref{CP_uot})  reduce to those of OT (\ref{CP_ot}).
\end{theorem}

\begin{proof}
    By the equivalence of norms in finite dimensional space and the convergence analysis in \cite{2009Proximal} it leads to 
    \begin{equation*}
        \lim_{k \rightarrow +\infty} \sum_{x\in\Om_h}\left( \lvert \eta_h^k(x) - \eta_h^*(x) \rvert + \lvert \varphi_h^k(x) - \varphi_h^*(x) \rvert \right) = 0,
    \end{equation*}
    where $(\boldm_h^*, \eta_h^*, \varphi_h^*)$ is a group of optimal solution to the UOT problem \eqref{duot1}. Therefore combining Theorem \ref{discrete convergence}, for any $\alpha > \frac{d L}{2}$ we have $\eta_h^* \equiv 0$ and $(\boldm_h^*, \varphi_h^*)$ is a pair of optimal solution to the OT problem \eqref{dot}. Correspondingly,
    \begin{equation}
        \lim_{k \rightarrow +\infty} \sum_{x\in\Om_h}\left( \lvert \eta_h^k(x) \rvert + \lvert \varphi_h^k(x) - \varphi_h^*(x) \rvert \right) = 0,
    \end{equation}
    which indicates that
    \begin{equation}
        \lim_{k \rightarrow +\infty} \lVert \eta_h^k \rVert_{h,\infty} + \lVert  \varphi_h^k - \varphi_h^* \rVert_{h,\infty} = 0,
    \end{equation}
    where the norm $\lVert u \rVert_{h,\infty}$ is defined as
    \[
        \lVert u \rVert_{h,\infty} = \max_{x\in\Om_h}\lvert u(x) \rvert, \quad \text{for $\forall u \in \R^{(N+1)^d}$}.
    \]
    Equivalently, one has that for any $\epsilon > 0$, there exists an integer $N > 0$ such that for any $k > N$,
    \begin{equation}\label{eps}
        \lVert \eta_h^{k - 1} \rVert_{h,\infty} + \lVert \varphi_h^{k - 1} - \varphi_h^* \rVert_{h,\infty} < \epsilon.
    \end{equation}
    Note that for the step $\eta_h^{k}$ in \eqref{CP_uot}, at each point $x \in \Om_h$ we have
    \begin{equation}
    \begin{aligned}
        \eta_h^{k}(x) & = \operatorname{Prox}_{\alpha\mu f}(\eta_h^{k - 1}(x) + \mu \varphi_h^{k - 1}(x)) = \underset{\eta_h}{\arg\min}~ \lvert\eta_h\rvert + \frac{1}{2\alpha\mu} \| \eta_h - (\eta_h^{k - 1}(x) + \mu \varphi_h^{k - 1}(x)) \|_{h,2}^2 \\
        & = \operatorname{sign}(\eta_h^{k - 1}(x) + \mu \varphi_h^{k - 1}(x))\cdot \max\left\{\lvert \eta_h^{k - 1}(x) + \mu \varphi_h^{k - 1}(x) \rvert - \alpha\mu , 0\right\}.
    \end{aligned}
    \end{equation}
    Define 
    \[M:= \max\left\{\frac{d L}{2}, \lVert \varphi_h^* \rVert_{h,\infty}\right\}
    \] and for any $\alpha > M$, let $\epsilon \leq \min\left\{1, \mu\right\} (\alpha - M)$ in \eqref{eps}. Then for any $k > N$ we obtain that
    \begin{equation}
    \begin{aligned}
        \lvert \eta_h^{k - 1}(x) + \mu \varphi_h^{k - 1}(x) \rvert & \leq \lVert \eta_h^{k-1} \rVert_{h,\infty} + \mu \lVert \varphi_h^{k-1} \rVert_{h,\infty} \\ 
        & \leq \lVert \eta_h^{k-1} \rVert_{h,\infty} + \mu\lVert \varphi_h^{k-1} - \varphi_h^* \rVert_{h,\infty} + \mu\lVert \varphi_h^* \rVert_{h,\infty}\\
        & \leq \max \left\{1, \mu\right\}(\lVert \eta_h^{k-1} \rVert_{h,\infty} + \lVert \varphi_h^{k-1} - \varphi_h^* \rVert_{h,\infty}) + \mu\lVert \varphi_h^* \rVert_{h,\infty}\\
        & \leq \max\left\{1, \mu\right\}\epsilon + \mu\lVert \varphi_h^* \rVert_{h,\infty} \leq \max\left\{1, \mu\right\}\epsilon + \mu M \leq \alpha\mu
    \end{aligned}
    \end{equation}
    which indicates that $\eta_h^{k}(x) \equiv 0$ for any $x \in \Om_h$. And correspondingly, the iterates of UOT \eqref{CP_uot} reduce to those of OT \eqref{CP_ot} for $k$ large enough.
\end{proof}

\section{Numerical Experiments}\label{sc:numerical experiments}
In this section we use two examples: shape deformation and color transfer problems  to illustrate the application of UOT and OT problems using the primal-dual hybrid algorithm discussed above.

\textbf{I. Shape Deformation~} The first  example is used to illustrate the convergence of UOT problem to OT problem. Particularly, we take $d = 2$, $\Om_h = [0,1]^2$ and the discrete distributions $\rho_h^0$, $\rho_h^1$ are the silhouettes of  cat images of same mass \cite{li2018parallel}, as shown in  Figure \ref{rho0} and Figure \ref{rho1}.   The size of both images is $256 \times 256$ and the algorithm is terminated when the primal-dual gap $R_h^k < 10^{-6}$ or the iteration number reaches $300000$. 
\begin{figure}[!h]
\centering
\begin{minipage}{0.47\linewidth}
    \centering
    \includegraphics[width = 1.0\textwidth]{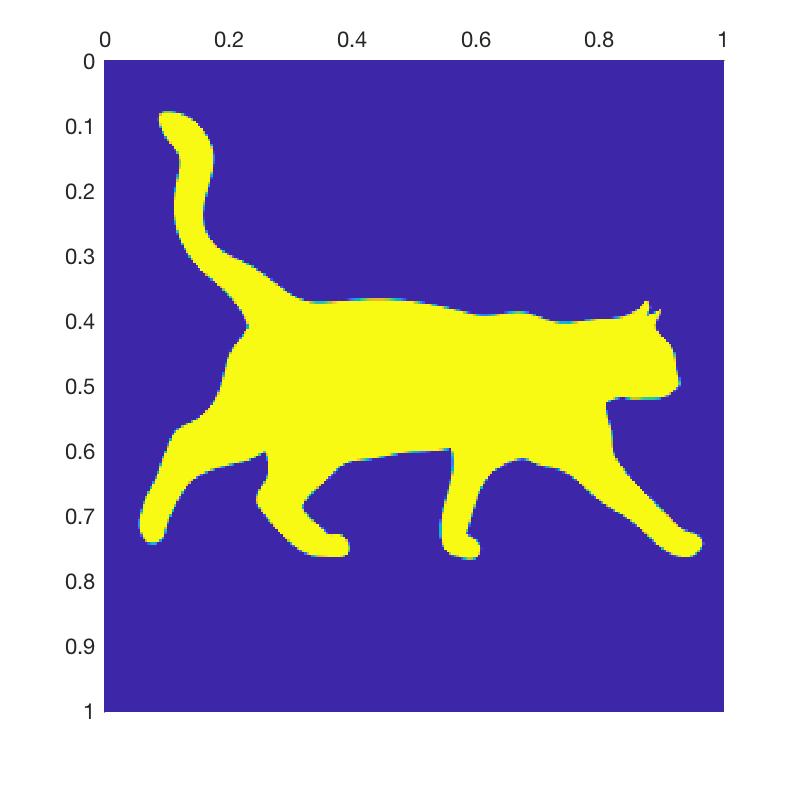}
    \caption{$\rho_h^0$}
    \label{rho0}
\end{minipage}
\hfill
\begin{minipage}{0.47\linewidth}
    \centering
    \includegraphics[width = 1.0\textwidth]{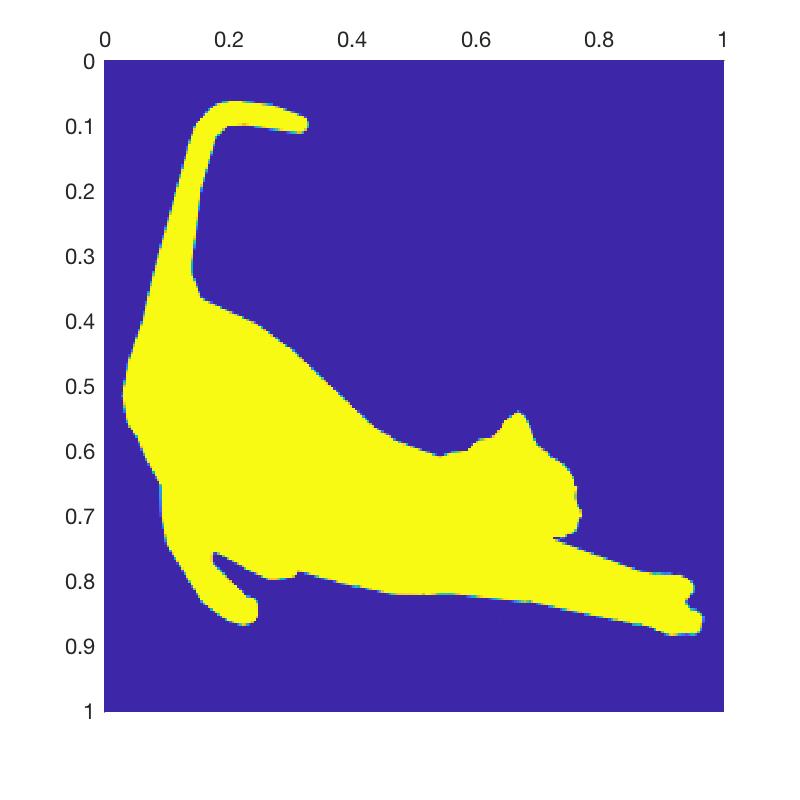}
    \caption{$\rho_h^1$}
    \label{rho1}
\end{minipage}
\end{figure}
\begin{itemize}
    \item \textbf{Convergence with $\alpha$.} ~We tune the value of $\alpha$ from $0.01$ to $1$ and get an optimal solution $(\boldm_\alpha, \eta_\alpha)$ for each $\alpha$ in UOT problem. Also we can obtain an optimal solution $\boldm_{\text{ot}}$ of the OT problem. As Theorem \ref{discrete convergence} states, the solution of the discrete UOT problem (\ref{duot1})  converges to the solution of discrete OT problem (\ref{dot}) as $\alpha$ gets larger, i.e  $\boldm_{\alpha} \rightarrow \boldm_{\text{ot}}$ and $\eta_{\alpha} \rightarrow 0$.  Figure \ref{difference figure} shows the difference  $m_{dif} = \lvert \boldm_{\alpha} - \boldm_{\text{ot}} \rvert_{h,2}$ and $\eta_{\text{dif}} = \lvert \eta_{\alpha} \rvert_{h,2} $ with different $\alpha$. For $\alpha = 0.6$, $m_{\text{dif}} = 3.2011 \times 10^{-8}$ and $\eta_{\text{dif}} = 0$, and when $\alpha$ is larger than some constant between $0.6$ and $0.7$, $\eta_{\text{dif}} = 0$ and $m_{\text{dif}} \sim 10^{-15}$, which is consistent with the results proved in Theorem \ref{discrete convergence}.
\begin{figure}[htbp]
\centering
\begin{minipage}{0.45\linewidth}
    \centering
    \includegraphics[width = 1.0\textwidth]{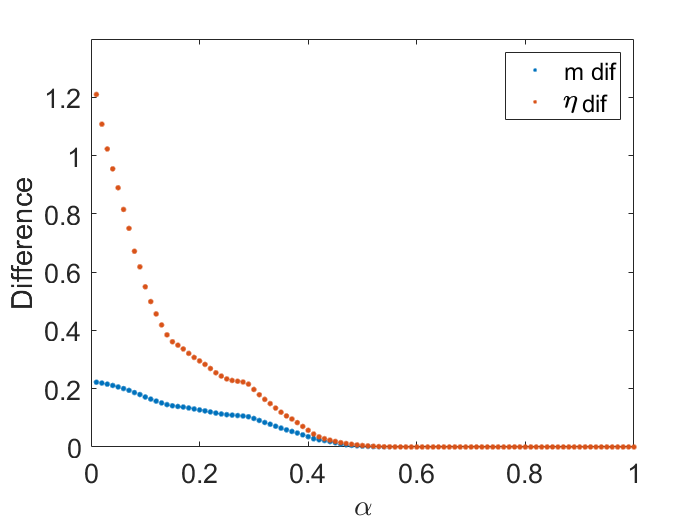}
\end{minipage}
\hfill
\begin{minipage}{0.45\linewidth}
    \centering
    \includegraphics[width = 1.0\textwidth]{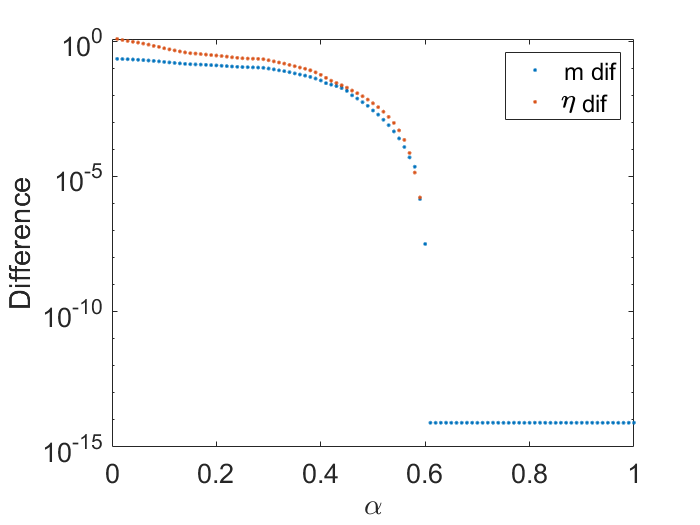}
\end{minipage}
\caption{The  figure shows the difference in  normal $y$-axis (left) and  $\log$ $y$-axis (right). Here we choose $N = 256$. It can be seen that both $m_{\text{dif}}$ and $\eta_{\text{dif}}$ go to 0 as $\alpha$ gets sufficient large. In particular when $\alpha = 0.6$, $m_{\text{dif}} = 3.2011 \times 10^{-8}$ and $\eta_{\text{dif}} = 0$, and when $\alpha$ is larger than some constant between $0.6$ and $0.7$, $\eta_{\text{dif}} = 0$ and $m_{\text{dif}} \sim 10^{-15}$, which is consistent to the results proved in Theorem \ref{discrete convergence}.
}
\label{difference figure}
\end{figure}

\item \textbf{Independence to the size of meshgrid $h$.}~ Notice that in Theorem \ref{discrete convergence}, the convergence of UOT to OT is unrelated to the choice of grid size $h$ for $\alpha > dL / 2$ the optimal solution of UOT is always convergent to OT. We note that here $dL/2=1$. In the experiments we choose four different sizes of images with $128 \times 128$, $256 \times 256$, $512\times 512$, $1024\times1024$ and  the length is unified to $1$, i.e.  $Nh = L = 1$. For different $\alpha = 0.1, 0.4, 0.6, 0.63, 1.0$, the results are listed in Table \ref{tb1}.
\begin{table}[htbp]
	\centering
	\begin{tabular}{|c|c|c|c|c|}
		\hline
		 & $N = 128$ & $N = 256$ & $N = 512$ & $N = 1024$ \\
		 \hline
		 $\alpha = 0.1$ & 0.3621 & 0.3583 & 0.3570 &  0.3563 \\
		 $\alpha = 0.4$ & 0.0296 & 0.0288 & 0.0285 & 0.0283 \\
		 $\alpha = 0.6$ & 0.1301 $\times 10^{-4}$ & 0.1048 $\times 10^{-4}$ & 0.1369 $\times 10^{-4}$ & 0.1515 $\times 10^{-4}$ \\
		 $\alpha = 0.63$ & 0.1248 $\times 10^{-7}$ & 0 & 0 & 0\\
		 $\alpha = 1.0$ & 0 & 0 & 0 & 0 \\
		 \hline
	\end{tabular}
\caption{The value of $\lvert\eta_{\text{dif}}^*\rvert_{h,2}$ with different $\alpha$ and different $N$(or $h$).}
\label{tb1}
\end{table}
As we can see, for any fixed $\alpha$  $\lvert\eta_{\text{dif}}^*\rvert_{h,2}$ remains  almost the same and when $\alpha$ is larger than $0.63$, which is small than $dL/2=1$, all $\lvert\eta_{\text{dif}}^*\rvert_{h,2}$ equal to $0$ for different $N$.

\end{itemize}

\textbf{II. Color Transfer~} Besides the transformation between shape images, we also provide an application of UOT model for color transfer between three-channel images.

The given target and source images are first transferred to the CIE-lab space $(l,a,b)$, where the $l$-space represents the luminance of the image,   and $a$ and $b$ are chromaticity coordinates. We fix the $l$-space as it is related to the lightness and normalize $a$ and $b$-components into $[0, 1]$. Then both $a$ and $b$-components are divided into $32$ intervals and the  color histograms is obtained on these intervals respectively.

For both $a$ and $b$-components, we solve both UOT and OT problems to get the optimal flux $m_h^*$. Then we can compute the velocity field with a value in $a$-space and $b$-space through the connection in Lemma \ref{prop:uot1} as follows:
\begin{equation}
	v(t;x) = \frac{m_h^*(x)}{t\rho_1(x) + (1 - t)\rho_0(x)} \triangleq \frac{m_h^*(x)}{\mu_t(x)},
\end{equation}
where $t$ is virtual time and $x \in [0, 1]$ is the partition position. In practice,  the supports of $\rho_0$ and $\rho_1$ are not always the same. To ensure the existence of $v(t;x)$ and eliminate the singularity of $v(t;x)$  caused by small $\mu_t$, we add a small perturbation $\epsilon$ on $\mu_t$, i.e.
\begin{equation}
	v_\epsilon(t;x) = \frac{m_h^*(x)}{t\rho_1(x) + (1 - t)\rho_0(x) + \epsilon} = \frac{m_h^*(x)}{\mu_t(x) + \epsilon}.
\end{equation}

For each $x_0$ sampled from $\rho_0$, the corresponding trasported $x_1$ in $\rho_1$ can be obtained by solving the following ODE:
\begin{equation}
	\left\{
		\begin{aligned}
			& \dot{x}(t) = v_\epsilon(t;x(t)), \quad t \in(0, 1), \\
			& x(0) = x_0,
		\end{aligned}
	\right.
\end{equation}  
and $x(1)$ is the target distribution.

By one-step forward Euler method, we obtain the formula:
\begin{equation}
	x(t_k) = x(t_{k - 1}) + \frac{m_h^*(x(t_{k - 1}))}{\mu_{t_{k-1}}(x(t_{k-1})) + \epsilon}, ~ \text{for $k = 1,2,\cdots, N$}, 
\end{equation}
where $t_k = k\Delta t$ and $N\Delta t  = 1$. 

Following this procedure, we transport the $a$ and $b$ components of every pixel in the target image to the corresponding one in the new image. Figure \ref{fig:colors} show the results of color transfer for three pair of images of size $512*512$.

\begin{figure}[H]
	\centering
	\begin{minipage}{0.15\linewidth}
		\centering
		\includegraphics[width=1.0\textwidth]{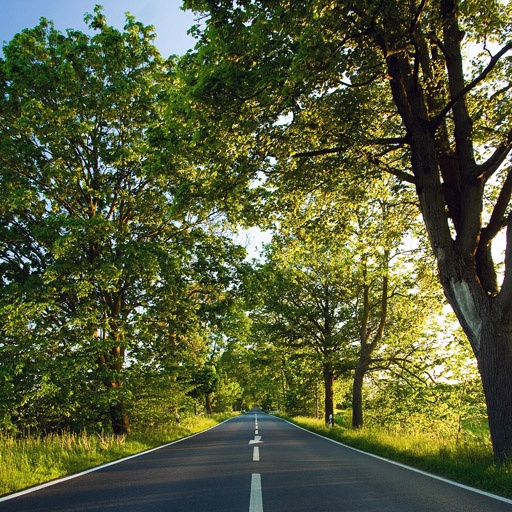}
		
	\end{minipage}
	\begin{minipage}{0.15\linewidth}
		\centering
		\includegraphics[width=1.0\textwidth]{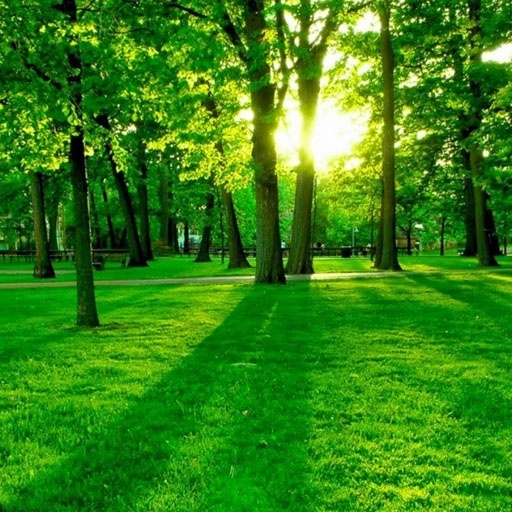}
		
	\end{minipage}
	\begin{minipage}{0.15\linewidth}
		\centering
		\includegraphics[width=1.0\textwidth]{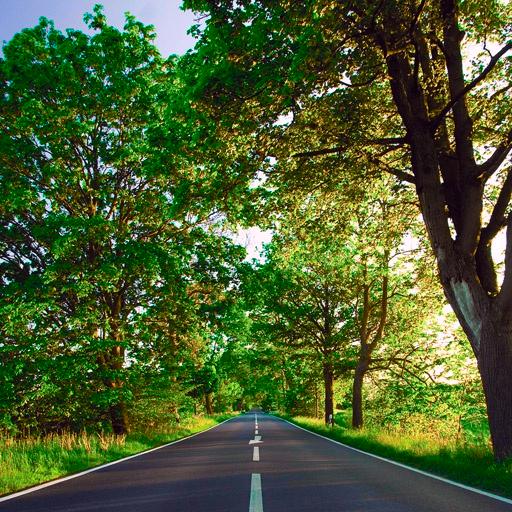}
	\end{minipage}
	\begin{minipage}{0.15\linewidth}
		\centering
		\includegraphics[width=1.0\textwidth]{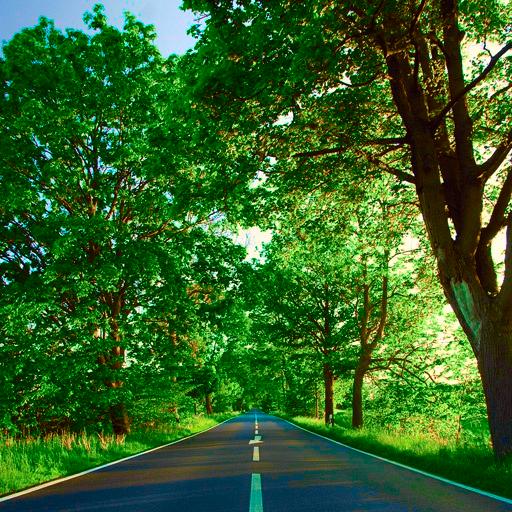}
	\end{minipage}
	\begin{minipage}{0.15\linewidth}
		\centering
		\includegraphics[width=1.0\textwidth]{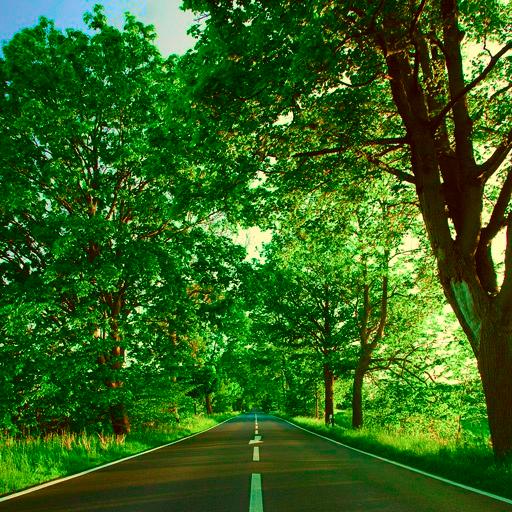}
	\end{minipage}
	\begin{minipage}{0.15\linewidth}
		\centering
		\includegraphics[width=1.0\textwidth]{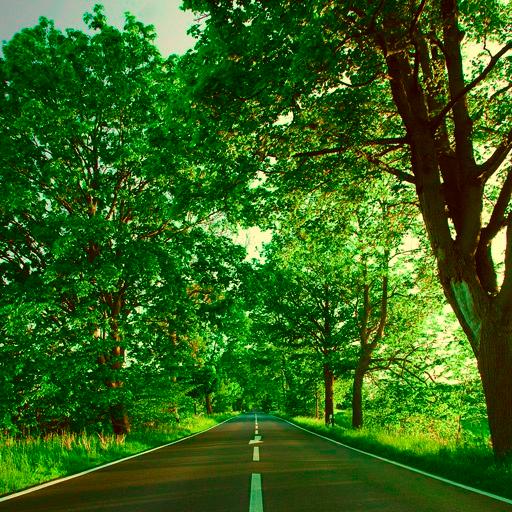}
	\end{minipage}
	\hfill
	\begin{minipage}{0.15\linewidth}
		\centering
		\includegraphics[width=1.0\textwidth]{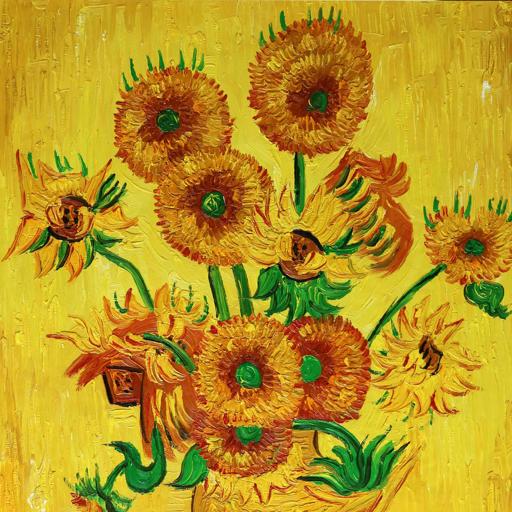}
	\end{minipage}
	\begin{minipage}{0.15\linewidth}
		\centering
		\includegraphics[width=1.0\textwidth]{figures/cf2.jpg}
	\end{minipage}
	\begin{minipage}{0.15\linewidth}
		\centering
		\includegraphics[width=1.0\textwidth]{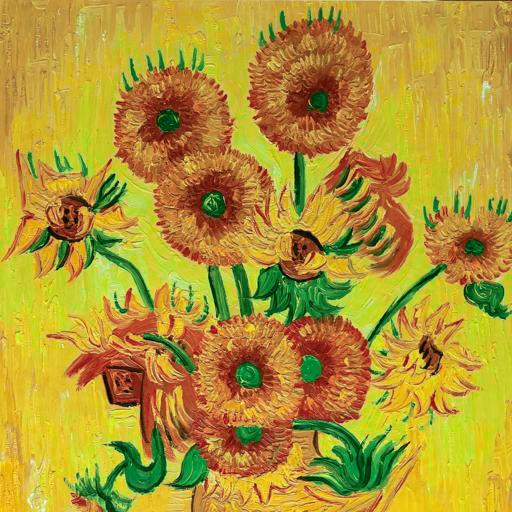}
	\end{minipage}
	\begin{minipage}{0.15\linewidth}
		\centering
		\includegraphics[width=1.0\textwidth]{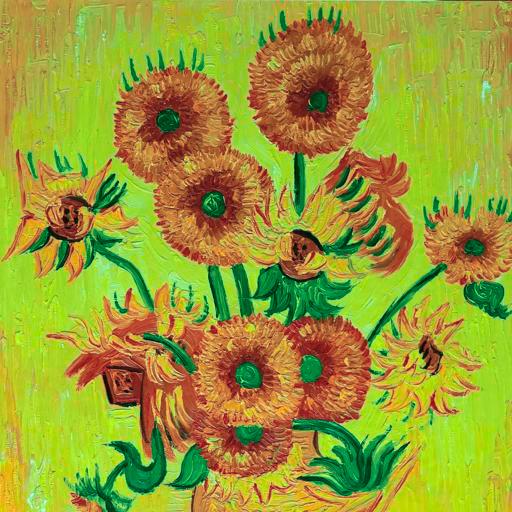}
	\end{minipage}
	\begin{minipage}{0.15\linewidth}
		\centering
		\includegraphics[width=1.0\textwidth]{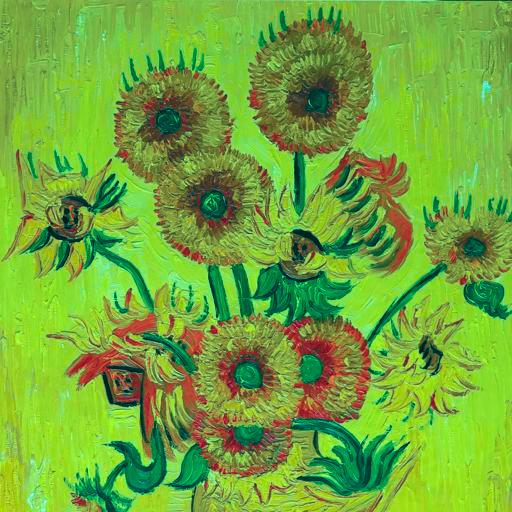}
	\end{minipage}
	\begin{minipage}{0.15\linewidth}
		\centering
		\includegraphics[width=1.0\textwidth]{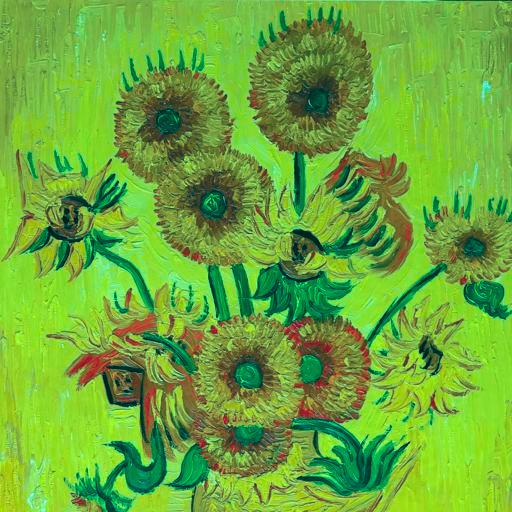}
	\end{minipage}
	\hfill
	\begin{minipage}{0.15\linewidth}
		\centering
		\includegraphics[width=1.0\textwidth]{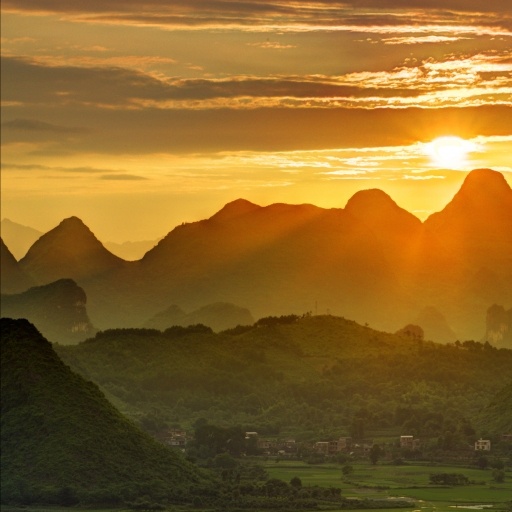}
		
	\end{minipage}
	\begin{minipage}{0.15\linewidth}
		\centering
		\includegraphics[width=1.0\textwidth]{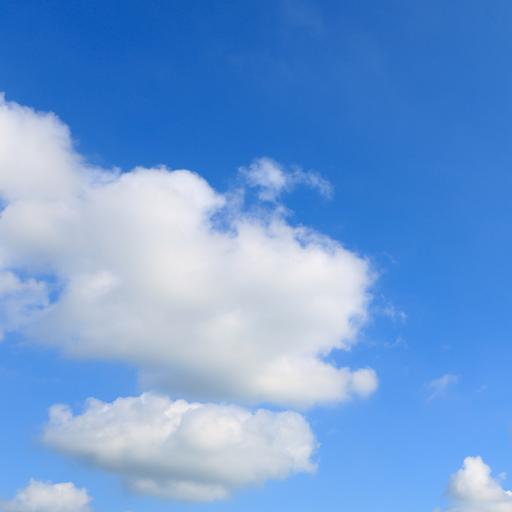}
		
	\end{minipage}
	\begin{minipage}{0.15\linewidth}
		\centering
		\includegraphics[width=1.0\textwidth]{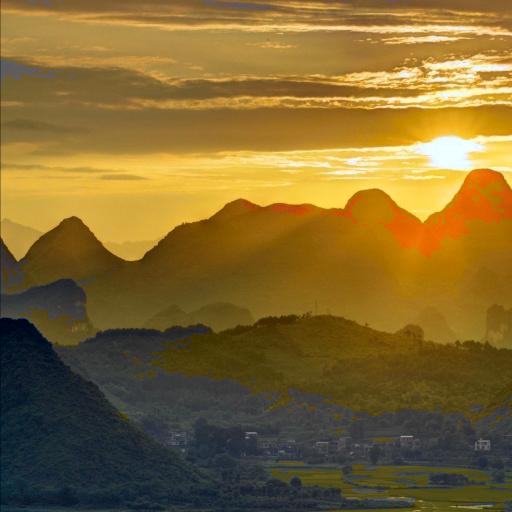}
	\end{minipage}
	\begin{minipage}{0.15\linewidth}
		\centering
		\includegraphics[width=1.0\textwidth]{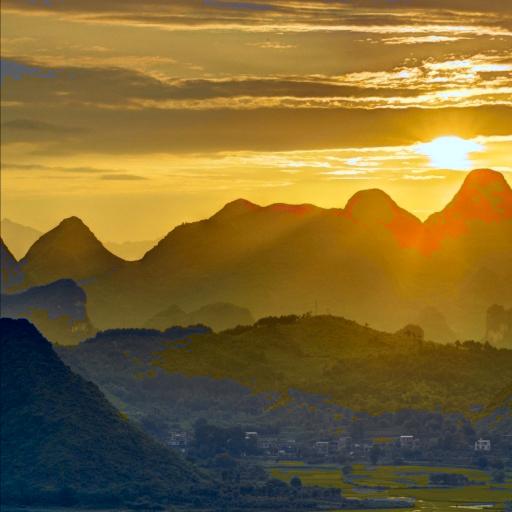}
	\end{minipage}
	\begin{minipage}{0.15\linewidth}
		\centering
		\includegraphics[width=1.0\textwidth]{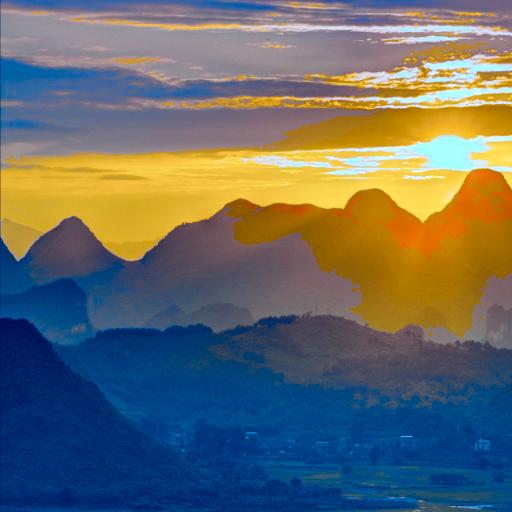}
	\end{minipage}
	\begin{minipage}{0.15\linewidth}
		\centering
		\includegraphics[width=1.0\textwidth]{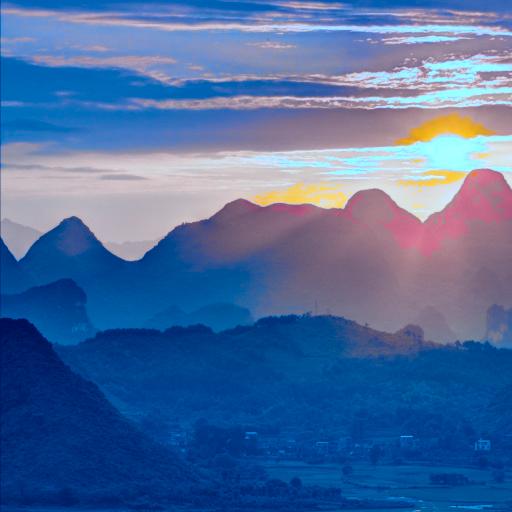}
	\end{minipage}
	\caption{In each row, the first two columns are the target and source images respectively. The last four  are the results by color transferring with different $\alpha = 0.05, 0.1, 0.2, 0.5$ respectively. As the value of $\alpha$ increases, the color distribution of target image are  partially transferred.}
\end{figure}
\label{fig:colors}
\section{Conclusions}
In this paper, we established the convergence from the Beckmann formulation of UOT to that of OT in both continuous and discrete  settings.  We proposed to apply a primal-dual hybrid algorithm for solving the UOT problem, and provide a lower bound for the regularization parameter $\alpha$ of UOT for  its solution reducing to the one of OT problem. Finally, we provide some applications of the UOT model and illustrate the convergence numerically.

\normalem
\bibliographystyle{unsrt}
\bibliography{ref.bib}

\begin{thebibliography}{10}

\bibitem{monge1781memoire}
Gaspard Monge.
\newblock M{\'e}moire sur la th{\'e}orie des d{\'e}blais et des remblais.
\newblock {\em Mem. Math. Phys. Acad. Royale Sci.}, pages 666--704, 1781.

\bibitem{kantorovich2006problem}
Leonid~Vitalevich Kantorovich.
\newblock On a problem of {Monge}.
\newblock {\em J. Math. Sci.(NY)}, 133:1383, 2006.

\bibitem{ferradans2014regularized}
Sira Ferradans, Nicolas Papadakis, Gabriel Peyr{\'e}, and Jean-Fran{\c{c}}ois
  Aujol.
\newblock Regularized discrete optimal transport.
\newblock {\em SIAM Journal on Imaging Sciences}, 7(3):1853--1882, 2014.

\bibitem{papadakis2015optimal}
Nicolas Papadakis.
\newblock {\em Optimal transport for image processing}.
\newblock PhD thesis, Universit{\'e} de Bordeaux; Habilitation thesis, 2015.

\bibitem{salimans2018improving}
Tim Salimans, Han Zhang, Alec Radford, and Dimitris Metaxas.
\newblock Improving {GAN}s using optimal transport.
\newblock {\em arXiv preprint arXiv:1803.05573}, 2018.

\bibitem{arjovsky2017wasserstein}
Martin Arjovsky, Soumith Chintala, and L{\'e}on Bottou.
\newblock Wasserstein generative adversarial networks.
\newblock In {\em International conference on machine learning}, pages
  214--223. PMLR, 2017.

\bibitem{frogner2015learning}
Charlie Frogner, Chiyuan Zhang, Hossein Mobahi, Mauricio Araya-Polo, and Tomaso
  Poggio.
\newblock Learning with a {Wasserstein} loss.
\newblock {\em arXiv preprint arXiv:1506.05439}, 2015.

\bibitem{trudinger2008monge}
Neil~S Trudinger and Xu-Jia Wang.
\newblock The {Monge-Ampere} equation and its geometric applications.
\newblock {\em Handbook of geometric analysis}, 1:467--524, 2008.

\bibitem{santambrogio2015optimal}
Filippo Santambrogio.
\newblock Optimal transport for applied mathematicians.
\newblock {\em Birk{\"a}user, NY}, 55(58-63):94, 2015.

\bibitem{de2012blue}
Fernando De~Goes, Katherine Breeden, Victor Ostromoukhov, and Mathieu Desbrun.
\newblock Blue noise through optimal transport.
\newblock {\em ACM Transactions on Graphics (TOG)}, 31(6):1--11, 2012.

\bibitem{villani2021topics}
C{\'e}dric Villani.
\newblock {\em Topics in optimal transportation}, volume~58.
\newblock American Mathematical Soc., 2021.

\bibitem{peyre2019computational}
Gabriel Peyr{\'e}, Marco Cuturi, et~al.
\newblock Computational optimal transport: With applications to data science.
\newblock {\em Foundations and Trends{\textregistered} in Machine Learning},
  11(5-6):355--607, 2019.

\bibitem{chizat2018unbalanced}
Lenaic Chizat, Gabriel Peyr{\'e}, Bernhard Schmitzer, and Fran{\c{c}}ois-Xavier
  Vialard.
\newblock Unbalanced optimal transport: {Dynamic and Kantorovich formulations}.
\newblock {\em Journal of Functional Analysis}, 274(11):3090--3123, 2018.

\bibitem{figalli2010new}
Alessio Figalli and Nicola Gigli.
\newblock A new transportation distance between non-negative measures, with
  applications to gradients flows with {Dirichlet} boundary conditions.
\newblock {\em Journal de math{\'e}matiques pures et appliqu{\'e}es},
  94(2):107--130, 2010.

\bibitem{liero2018optimal}
Matthias Liero, Alexander Mielke, and Giuseppe Savar{\'e}.
\newblock Optimal entropy-transport problems and a new {Hellinger--Kantorovich}
  distance between positive measures.
\newblock {\em Inventiones mathematicae}, 211(3):969--1117, 2018.

\bibitem{gangbo2019unnormalized}
Wilfrid Gangbo, Wuchen Li, Stanley Osher, and Michael Puthawala.
\newblock Unnormalized optimal transport.
\newblock {\em Journal of Computational Physics}, 399:108940, 2019.

\bibitem{piccoli2014generalized}
Benedetto Piccoli and Francesco Rossi.
\newblock Generalized {Wasserstein} distance and its application to transport
  equations with source.
\newblock {\em Archive for Rational Mechanics and Analysis}, 211(1):335--358,
  2014.

\bibitem{li2022application}
Da~Li, Michael~P Lamoureux, and Wenyuan Liao.
\newblock Application of an unbalanced optimal transport distance and a mixed
  l1/wasserstein distance to full waveform inversion.
\newblock {\em Geophysical Journal International}, 230(2):1338--1357, 2022.

\bibitem{zhou2018wasserstein}
DT~Zhou, JING Chen, H~Wu, DH~Yang, and LY~Qiu.
\newblock The wasserstein-fisher-rao metric for waveform based earthquake
  location.
\newblock {\em arXiv preprint arXiv:1812.00304}, 2018.

\bibitem{pmlr-v139-fatras21a}
Kilian Fatras, Thibault Sejourne, R{\'e}mi Flamary, and Nicolas Courty.
\newblock Unbalanced minibatch optimal transport; applications to domain
  adaptation.
\newblock In Marina Meila and Tong Zhang, editors, {\em Proceedings of the 38th
  International Conference on Machine Learning}, volume 139 of {\em Proceedings
  of Machine Learning Research}, pages 3186--3197. PMLR, 18--24 Jul 2021.

\bibitem{bernton2017inference}
Espen Bernton, Pierre~E Jacob, Mathieu Gerber, and Christian~P Robert.
\newblock Inference in generative models using the {W}asserstein distance.
\newblock {\em arXiv preprint arXiv:1701.05146}, 1(8):9, 2017.

\bibitem{pmlr-v97-dukler19a}
Yonatan Dukler, Wuchen Li, Alex Lin, and Guido Montufar.
\newblock {W}asserstein of {W}asserstein loss for learning generative models.
\newblock In Kamalika Chaudhuri and Ruslan Salakhutdinov, editors, {\em
  Proceedings of the 36th International Conference on Machine Learning},
  volume~97 of {\em Proceedings of Machine Learning Research}, pages
  1716--1725. PMLR, 09--15 Jun 2019.

\bibitem{beckmann1952continuous}
Martin Beckmann.
\newblock A continuous model of transportation.
\newblock {\em Econometrica: Journal of the Econometric Society}, pages
  643--660, 1952.

\bibitem{benamou2000computational}
Jean-David Benamou and Yann Brenier.
\newblock A computational fluid mechanics solution to the {Monge-Kantorovich}
  mass transfer problem.
\newblock {\em Numerische Mathematik}, 84(3):375--393, 2000.

\bibitem{barrett2009partial}
John~W Barrett and Leonid Prigozhin.
\newblock Partial {$L_1$ Monge-Kantorovich} problem: Variational formulation
  and numerical approximation.
\newblock {\em Interfaces and Free Boundaries}, 11(2):201--238, 2009.

\bibitem{caffarelli2010free}
Luis~A Caffarelli and Robert~J McCann.
\newblock Free boundaries in optimal transport and {Monge-Ampere} obstacle
  problems.
\newblock {\em Annals of mathematics}, pages 673--730, 2010.

\bibitem{figalli2010optimal}
Alessio Figalli.
\newblock The optimal partial transport problem.
\newblock {\em Archive for rational mechanics and analysis}, 195(2):533--560,
  2010.

\bibitem{chizat2018interpolating}
Lenaic Chizat, Gabriel Peyr{\'e}, Bernhard Schmitzer, and Fran{\c{c}}ois-Xavier
  Vialard.
\newblock An interpolating distance between optimal transport and {Fisher-Rao}
  metrics.
\newblock {\em Foundations of Computational Mathematics}, 18(1):1--44, 2018.

\bibitem{kondratyev2016new}
Stanislav Kondratyev, L{\'e}onard Monsaingeon, and Dmitry Vorotnikov.
\newblock A new optimal transport distance on the space of finite {Radon}
  measures.
\newblock {\em Advances in Differential Equations}, 21(11/12):1117--1164, 2016.

\bibitem{li2018parallel}
Wuchen Li, Ernest~K Ryu, Stanley Osher, Wotao Yin, and Wilfrid Gangbo.
\newblock A parallel method for earth mover’s distance.
\newblock {\em Journal of Scientific Computing}, 75(1):182--197, 2018.

\bibitem{doi:10.1137/18M1219813}
Jialin Liu, Wotao Yin, Wuchen Li, and Yat~Tin Chow.
\newblock Multilevel optimal transport: A fast approximation of {Wasserstein}-1
  distances.
\newblock {\em SIAM Journal on Scientific Computing}, 43(1):A193--A220, 2021.

\bibitem{ennaji2020beckmann}
Hamza Ennaji, Noureddine Igbida, and Van Nguyen.
\newblock Beckmann-type problem for degenerate {Hamilton-Jacobi} equations.
\newblock 2020.

\bibitem{EZC10}
Ernie Esser, Xiaoqun Zhang, and Tony~F. Chan.
\newblock A general framework for a class of first order primal-dual algorithms
  for convex optimization in imaging science.
\newblock {\em SIAM Journal on Imaging Sciences}, 3(4):1015--1046, 2010.

\bibitem{chambolle2011first}
Antonin Chambolle and Thomas Pock.
\newblock A first-order primal-dual algorithm for convex problems with
  applications to imaging.
\newblock {\em Journal of mathematical imaging and vision}, 40(1):120--145,
  2011.

\bibitem{braides2002gamma}
Andrea Braides et~al.
\newblock {\em Gamma-convergence for Beginners}, volume~22.
\newblock Clarendon Press, 2002.

\bibitem{buttazzo1982gamma}
Giuseppe Buttazzo and Gianni Dal~Maso.
\newblock {$\Gamma$}-convergence and optimal control problems.
\newblock {\em Journal of optimization theory and applications},
  38(3):385--407, 1982.

\bibitem{brezis1973semi}
Ha{\"\i}m Br{\'e}zis and Walter~A Strauss.
\newblock Semi-linear second-order elliptic equations in l1.
\newblock {\em Journal of the Mathematical Society of Japan}, 25(4):565--590,
  1973.

\bibitem{2009Proximal}
B.~S. He, X.~L. Fu, and Z.~K. Jiang.
\newblock Proximal-point algorithm using a linear proximal term.
\newblock {\em Journal of Optimization Theory and Applications},
  141(2):299--319, 2009.

\end{thebibliography}

\newpage
\section*{Appendix}

\appendix
\section{Proofs of $\Gm$-convergence Lemmas and Theorems}\label{apdx}
\begin{lemma}[Proposition \ref{connection}]
    It holds that
    \begin{gather}
    \begin{split}
    & \Gm_{\seq}(\N^+, X^-)\lim_n f_n =\Gamma\hbox{-}\limsup_{n\to\infty} f_n ,\\
    & \Gm_{\seq}(\N^-, X^-)\lim_n f_n = \Gamma\hbox{-}\liminf_{n\to\infty} f_n.
    \end{split}
    \end{gather}
    Consequently, if $f:=\Gm_{\seq}(\N, X^-)f_n$ exists, then 
    $f_n$ $\Gm$-converges to $f$.
\end{lemma}

\begin{proof}
    We  give the proof of the first equality
    \begin{equation}
        \Gm_{\seq}(\N^+, X^-)\lim_n f_n=\Gamma\hbox{-}\limsup_{n\to\infty} f_n
    \end{equation}
    while the second one can be proved similarly. Equivalently, our goal can be rewritten as
    \begin{equation}
        \inf_{x^n \rightarrow x} \limsup_{n\rightarrow \infty} f_n(x_n) = \sup_{N_x}  \limsup_{n\rightarrow \infty} \inf_{y \in N_x} f_n(y),
    \end{equation}
    where $N_x$ ranges over all the neighborhoods. Then for $\forall \epsilon > 0$, there exists a neighborhood $N_x$ such that
    \begin{equation}
        \nlimsup \inf_{y\in N_x} f_n(y) > \Gamma\hbox{-}\limsup_{n\to\infty} f_n - \epsilon.
    \end{equation}
    On the other hand, for $\forall \epsilon > 0$ there exists a sequence $(x^n)$ converging to $x$ such that
    \begin{equation}
        \nlimsup f_n(x_n) < \Gm_{\seq}(\N^+, X^-) + \epsilon.
    \end{equation}
    Since for any $n$ large enough we have $x_n \in N_x$, then we can get the following relationship
    \begin{equation}
        f_n(x_n) \geq \inf_{y\in N_x} f_n(y)
    \end{equation}
    for all $n$ large enough. Then take $\limsup$ from both sides and we obtain that
    \begin{equation}
        \nlimsup f_n(x_n) \geq \nlimsup \inf_{y \in N_x} f_n(y).
    \end{equation}
    Combining with the above two inequalities, it is obvious that
    \begin{equation}
        \Gm_{\seq}(\N^+, X^-) + \epsilon > \nlimsup f_n(x^n) \geq \nlimsup \inf_{y \in N_x} f_n(y) > \Gamma\hbox{-}\limsup_{n\to\infty} f_n - \epsilon.
    \end{equation}
    As  $\epsilon$ is arbitrary, we have 
    \begin{equation}\label{eq:ineq_left}
        \Gm_{\seq}(\N^+, X^-)\lim_n f_n \geq \Gamma\hbox{-}\limsup_{n\to\infty} f_n.
    \end{equation}
    On the other hand, for $\forall n > 0$, suppose $B_n = N(x, \frac{1}{n})$ and we have 
    \begin{equation}
        \Gamma\hbox{-}\limsup_{n\to\infty} f_n \geq \nlimsup\inf_{y\in B_n} f_n(y).
    \end{equation}
    For any fixed $n$ and $\forall \epsilon > 0$, there exists $x^n \in B_n$ such that
    \begin{equation}
        \inf_{y \in B_n} f_n(y) > f_n(x^n) - \epsilon,
    \end{equation}
    By taking $\limsup$ from both sides we obtain
    \begin{equation}
        \nlimsup \inf_{y \in B_n} f_n(y) \geq \nlimsup f_n(x^n) - \epsilon.
    \end{equation}
    Note that $x^n \in B_n$ for any $n$, then $x^n \rightarrow x$ as $n$ goes to infinity. We can further get
    \begin{equation}
        \nlimsup f_n(x^n) \geq \inf_{x^n \rightarrow x} \nlimsup f_n(x^n) = \Gm_{\seq}(\N^+, X^-)\lim_n f_n.
    \end{equation}
    Therefore
    \begin{equation}
        \Gamma\hbox{-}\limsup_{n\to\infty} f_n \geq \nlimsup\inf_{y\in B_n} f_n(y)\geq \nlimsup f_n(x^n) - \epsilon\geq \Gm_{\seq}(\N^+, X^-)\lim_n f_n - \epsilon.
    \end{equation}
    
  This leads to
    \begin{equation}\label{eq:ineq_right}
        \Gamma\hbox{-}\limsup_{n\to\infty} f_n \geq \Gm_{\seq}(\N^+, X^-)\lim_n f_n.
    \end{equation}
Combining  the two inequality \eqref{eq:ineq_left} and \eqref{eq:ineq_right}, we obtain
    \begin{equation}
        \Gm_{\seq}(\N^+, X^-)\lim_n f_n=\Gamma\hbox{-}\limsup_{n\to\infty} f_n.
    \end{equation}
    
\end{proof}

\begin{lemma}[Lemma \ref{lemma1}]
	Let $X$ be a topological space, and let $(f_n)$ be a sequence of functionals from $X$ into $\bar{\R}= [-\infty, +\infty]$. If
	\begin{equation}
	    \Gm_{\seq}(\N, X^-)\lim_n f_n = f,
	\end{equation}
	then 
	\begin{equation}
	    \inf_X f \geq \limsup_n[\inf_X f_n].
	\end{equation}
	Moreover, if there exists a sequence $(x^n)$ converging to $x_0$ in $X$, with
	\begin{equation}
	    \liminf_n f_n(x^n) = \liminf_n [\inf_X f_n],
	\end{equation}
	then
	\begin{equation}
	    f(x_0) = \inf_X f = \lim_n[\inf_X f_n].
	\end{equation}
\end{lemma}

\begin{proof}
	By the definition of $\inf_{X}$, for any $\epsilon > 0$, there exists $x \in X$ such that
	\begin{equation}
		f(x)  < \inf\limits_{x \in X} f +\epsilon.
	\end{equation}
	Using the fact that 
	\begin{equation}
		f = \Gamma_{\seq}(\N, X^-)\lim_n f_n,
	\end{equation}
	we have for $\forall x \in X$
	\begin{equation}
		f(x) = \inf_{x^n \rightarrow x} \limsup_n f_n(x^n) < \inf\limits_{x \in X} f +\epsilon.
	\end{equation}
	Therefore, there exists a sequence $(x^n)$ converging to $x$ such that
	\begin{equation}
		\limsup_n f_n(x^n) < \inf_{x \in X} f + 2\epsilon,
	\end{equation}
	and
	\begin{equation}
		\limsup_n (\inf_{x \in X} f_n) \leq \limsup_n f_n(x^n) < \inf_{x \in X} f + \epsilon.
	\end{equation}
	For  $\epsilon$ is arbitrary, we get the first part of the lemma:
	\begin{equation}
		\limsup_n(\inf_{x \in X} f_n) \leq \inf_{x \in X} f.
	\end{equation}
	Moreover, if there exist a sequence $(x^n)$ such that $ x^n \rightarrow x_0$ in $X$, with
	\begin{equation}
		\liminf_n f_n(x^n) = \liminf_n (\inf_{x \in X} f_n),
	\end{equation}
	then we can have 
	\begin{equation}
	\begin{aligned}
	    \inf_{x \in X} f \leq f(x_0) = \inf_{\bar{x}^n \rightarrow x_0} \liminf_n f_n(\bar{x}^n) & \leq \liminf_n f_n(x_n) = \liminf_n(\inf_{x \in X} f_n) \\
		& \leq \limsup(\inf_{x \in X} f_n) \leq \inf_{x \in X} f.
	\end{aligned}
	\end{equation}
	Therefore, all the inequalities are equal, i.e.
	\begin{equation}
		f(x_0) = \inf_{x \in X} f = \lim_n (\inf_{x \in X}f_n),
	\end{equation}
	where the last equality holds because $\underset{h}{\liminf}(\underset{x\in X}{\inf} f_n) = \underset{h}{\limsup}(\underset{x\in X}{\inf} f_n)  = \underset{h}{\lim}(\underset{x\in X}{\inf} f_n)$. 
\end{proof}

\begin{lemma}[Lemma \ref{lemma2}]
	Let $X, Y$ be two topological space; let $(f_n), (g_n)$ be two sequences of functionals from the product space $X \times Y$ to $\bar{\R}^{+} = [0, +\infty]$, and let $(x_0, y_0) \in X \times Y$. Suppose there exists $a, b \in \R^+$ such that
	\begin{equation}
	    \Gm_{\seq} (\N, X^-, Y) \lim_n f_n(x_0, y_0) = a,
	\end{equation}
	\begin{equation}
	    \Gm_{\seq} (\N, X, Y^-) \lim_n g_n(x_0, y_0) = b.
	\end{equation}
	Then it holds that
	\begin{equation}
	    \Gamma_{\seq} (\N, X^-, Y^-)\lim_n (f_n + g_n) (x_0,y_0) = a + b.
	\end{equation}
\end{lemma}

\begin{proof}
	According to the definition of $\Gm$-limits, we have
	\begin{equation}
		\Gm_{\seq}(\N^+, X^-, Y)\lim_n f_n(x_0, y_0) = \inf_{x^n \rightarrow x_0} \sup_{y^n \rightarrow y_0} \limsup_{n} f_n(x^n, y^n) = a,
	\end{equation}
	\begin{equation}
		\Gm_{\seq} (\N^+, X, Y^-) \lim_n g_n(x_0, y_0) = \sup_{x^n \rightarrow x_0} \inf_{y^n \rightarrow y_0} \limsup_n g_n(x^n, y^n) = b,
	\end{equation}
	\begin{equation}
		\Gm_{\seq} (\N^+, X^-, Y^-) \lim_n(f_n + g_n)(x_0, y_0) = \inf_{x^n \rightarrow x_0} \inf_{y^n \rightarrow y_0} \limsup_n (f_n(x^n, y^n) + g_n(x^n, y^n)).
	\end{equation}
	Notice that for any $\{(x^n, y^n)\} \in X \times Y$,
	\begin{equation}
		\limsup_n (f_n + g_n)(x^n, y^n) \leq \limsup_n f_n(x^n, y^n) + \limsup_n g_n(x^n, y^n),
	\end{equation}
	For any $\epsilon > 0$, there exists $\bar{y}^n \rightarrow y_0$ such that
	\begin{equation}
		\limsup_n g_n(x^n, \bar{y}^n) \leq \inf_{y^n \rightarrow y_0} \limsup_n g_n(x^n, y^n) + \epsilon.
	\end{equation}
	Then we have that
	\begin{equation}
		\begin{aligned}
			& \inf_{y^n \rightarrow y_0} \limsup_n (f_n(x^n, y^n) + y_n(x^n, y^n)) \\
			& \leq \inf_{y^n \rightarrow y_0}[\limsup_n f_n(x^n, y^n) + \limsup_n g_n(x^n, y^n)] \\
			& \leq \limsup_n f_n(x^n, \bar{y}^n) + \limsup_n g_n(x^n, \bar{y}^n)\\
			&\leq \sup_{y^n \rightarrow y_0} \limsup_n f_n(x^n, y^n) + \inf_{y^n \rightarrow y_0} \limsup_n g_n(x^n, y^n) + \epsilon.
		\end{aligned}
	\end{equation}
	Since $\epsilon$ is arbitrary, then we obtain that
	\begin{equation}
	\begin{aligned}
	    & \inf_{y^n \rightarrow y_0} \limsup_n (f_n(x^n, y^n) + g_n(x^n, y^n)) \\
	    & \leq \sup_{y^n \rightarrow y_0} \limsup_n f_n(x^n, y^n) + \inf_{y^n \rightarrow y_0} \limsup_n g_n(x^n, y^n).
	\end{aligned}
	\end{equation}
	Similarly, for $x^n$ we also have
	\begin{equation}
	\begin{aligned}
    	& \inf_{x^n \rightarrow x_0} \inf_{y^n \rightarrow y_0} \limsup_n (f_n(x^n, y^n) + g_n(x^n, y^n))\\ \leq 
    	& \inf_{x^n \rightarrow x_0} \sup_{y^n \rightarrow y_0} \limsup_{n} f_n(x^n, y^n) + \sup_{x^n \rightarrow x_0} \inf_{y^n \rightarrow y_0} \limsup_n g_n(x^n, y^n),
	\end{aligned}
	\end{equation}
	which leads to
	\begin{equation}
		\Gm_{\seq} (\N^+, X^-, Y^-) \lim_n(f_n + g_n)(x_0, y_0) \leq a + b.
	\end{equation}
	On the other hand, by the definition we can change the $\sup$ to the $\inf$ and get that
	\begin{equation}
		\Gm_{\seq}(\N^-, X^-, Y)\lim_n f_n(x_0, y_0) = \inf_{x^n \rightarrow x_0} \inf_{y^n \rightarrow y_0} \liminf_{n   } f_n(x^n, y^n) = a,
	\end{equation}
	\begin{equation}
		\Gm_{\seq} (\N^-, X, Y^-) \lim_n g_n(x_0, y_0) = \inf_{x^n \rightarrow x_0} \inf_{y^n \rightarrow y_0} \liminf_n g_n(x^n, y^n) = b,
	\end{equation}
	\begin{equation}
		\Gm_{\seq} (\N^-, X^-, Y^-) \lim_n(f_n + g_n)(x_0, y_0) = \inf_{x^n \rightarrow x_0} \inf_{y^n \rightarrow y_0} \liminf_n (f_n(x^n, y^n) + g_n(x^n, y^n)).
	\end{equation}
	It is obvious that
	\begin{equation}
		\begin{aligned}
			& \inf_{y^n \rightarrow y_0} \liminf_n (f_n(x^n, y^n) + g_n(x^n, y^n)) \\
			& \geq \inf_{y^n \rightarrow y_0}[\liminf_n f_n(x^n, y^n) + \liminf_n g_n(x^n, y^n)] \\
			& \geq \inf_{y^n \rightarrow y_0} \liminf_n f_n(x^n, y^n) + \inf_{y^n \rightarrow y_0} \liminf_n g_n(x^n, y^n),
		\end{aligned}
	\end{equation}
	therefore we obtain
	\begin{equation}
		\begin{aligned}
		& \inf_{x^n \rightarrow x_0} \inf_{y^n \rightarrow y_0} \liminf_n (f_n(x^n, y^n) + g_n(x^n, y^n)) \\ & \geq \inf_{x^n \rightarrow x_0}[\inf_{y^n \rightarrow y_0} \liminf_n f_n(x^n, y^n) + \inf_{y^n \rightarrow y_0} \liminf_n g_n(x^n, y^n)] \\
		& \geq \inf_{x^n \rightarrow x_0} \inf_{y^n \rightarrow y_0} \liminf_n g_n(x^n, y^n) + \inf_{x^n \rightarrow x_0} \inf_{y^n \rightarrow y_0} \liminf_{n} f_n(x^n, y^n) \\
		& = a + b.
		\end{aligned}
	\end{equation}
	Combining the two inequalities proved before, we obtain
	\begin{equation}
		\Gm_{\seq} (\N, X^-, Y^-) \lim_n(f_n + g_n)(x_0, y_0) = a + b.
	\end{equation} 
\end{proof}

\begin{lemma}[Lemma \ref{lemma3}]
    Suppose $\{E_n\}$ is a sequence of sets in space $X \times Y$. If there exists a set $E_\infty \in X \times Y$ satisfying the following two conditions:
    \begin{itemize}
        \item If $x^n \rightarrow x$, $y^n \rightarrow y$ and $(x^n, y^n) \in E_n$ for infinitely many $n$, then $(x, y) \in E_\infty$;
        \item If $(x,y) \in E_\infty$ and $x^n \rightarrow x$, then there exists $y^n \rightarrow y$ such that $(x^n, y^n) \in E_n$ for $n$ large enough,
    \end{itemize}
    then $\one_{E_\infty} = \Gm_{\seq}(\N, X, Y^-) \underset{n}{\lim} \one_{E_n}$.
\end{lemma}

\begin{proof}
    First by the second condition, for $\forall (x, y) \in E_\infty$ and $x^n \rightarrow x$, there exists $y^n \rightarrow y$ and $(x^n, y^n) \in E_n$ for large $n$, therefore $\one_{E_n}(x^n, y^n) = 1$ and we have 
    \begin{equation}
        \Gm_{\seq}(\N, X, Y^-) \lim_n \one_{E_n}(x, y) = \inf_{x^n \rightarrow x} \inf_{y^n \rightarrow y} \liminf_n \one_{E_n}(x^n, y^n) = 1 = \one_{\infty}(x, y).
    \end{equation}
    
    On the other hand, suppose $(x,y)$ is an arbitrary point not in $ E_\infty$. Then by the first condition, for any sequence $\{x^n\}$ and $\{y^n\}$ with limits equal to $x$ and $y$ respectively, $(x^n, y^n) \in E_n$ holds for only finite many $n$, otherwise $(x,y) \in E_\infty$.  Therefore we obtain that
    \begin{equation}
        \Gm_{\seq}(\N, X, Y^-) \lim_n \one_{E_n}(x, y) = \sup_{x^n\rightarrow}\inf_{y^n \rightarrow y}\limsup_n \one_{E_n}(x^n, y^n) = +\infty = \one_{E_\infty}(x, y).
    \end{equation}
    
    Therefore we conclude that $\one_{E_\infty} = \Gm_{\seq}(\N, X, Y^-) \underset{n}{\lim} \one_{E_n}$. 
\end{proof}

\begin{theorem}[Theorem \ref{main}]
	Suppose $X$ and $Y$ are two topological spaces, $(J_n)$ is a $\Gm$-convergent sequence of functionals defined on the product space $X \times Y$ whose $\Gm$-limits is denoted by $J_\infty$. Let $(E_n)$ be a sequence of set in $X \times Y$ and the sequence of indicator functions $(\one_{E_n})$ is also $\Gm$-convergent with $\Gm$-limit equal to $\one_{E_\infty}$, that is
	\begin{equation}
		J_\infty = \Gamma_{\seq}(\N, X^-, Y)\lim_n J_n,
	\end{equation}
	\begin{equation}
		\one_{E_\infty} = \Gamma_{\seq}(\N, X, Y^-)\lim_n \one_{E_n}.
	\end{equation}
    And for every $h \in \N^+$, $(x_n, y_n)$ is an optimal pair of the control problem
	\begin{equation}
		\min_{X \times Y} (J_n + \one_{E_n}).
	\end{equation}
	If $x_n \rightarrow x_\infty$ in $X$ and $y_n \rightarrow y_\infty$ in $Y$, then $(x_\infty, y_\infty)$ is an optimal pair of the control problem
	\begin{equation}
		\min_{X \times Y} (J_\infty + \one_{E_\infty}).
	\end{equation}
\end{theorem}

\begin{proof}
	First we use Lemma \ref{lemma3}. Since
	\begin{equation}
		J_\infty = \Gamma_{\seq}(\N, X^-, Y)\lim_n J_n,
	\end{equation}
	\begin{equation}
		\one_{E_\infty} = \Gamma_{\seq}(\N, X, Y^-)\lim_n \one_{E_n},
	\end{equation}
	then we have
	\begin{equation}
		J_\infty + \one_{E_\infty} = \Gm_{\seq}(\N, X^-, Y^-) \lim_n(J_n + \one_{E_n}).
	\end{equation}
	Notice that
	\begin{equation}
		\begin{aligned}
			\Gm_{\seq}(\N, X^-, Y^-) \lim_n(J_n + \one_{E_n}) & = \inf_{x^n \rightarrow x} \inf_{y^n \rightarrow y} \liminf_n (J_n + \one_{E_n}) (x^n, y^n) \\
			& = \inf_{\substack{x^n \rightarrow x \\ y^n \rightarrow y}}\liminf_n (J_n + \one_{E_n}) (x^n, y^n)\\
			& = \inf_{(x^n, y^n) \rightarrow (x, y)}\liminf_n (J_n + \one_{E_n}) (x^n, y^n)\\
			& = \inf_{(x^n, y^n) \rightarrow (x, y)}\limsup_n (J_n + \one_{E_n}) (x^n, y^n)\\
			& = \Gm_{\seq}(\N, (X \times Y)^-) \lim_n(J_n + \one_{E_n}),
		\end{aligned}		
	\end{equation}
	therefore we obtain that
	 \begin{equation}
	 	J_\infty + \one_{E_\infty} = \Gm_{\seq}(\N, (X \times Y)^-) \lim_n(J_n + \one_{E_n}).
	 \end{equation}
 	Moreover, from $x^n \rightarrow x_\infty$ in $X$ and $y^n \rightarrow y_\infty$ in $Y$ we can get $(x^n, y^n) \rightarrow (x_\infty, y_\infty)$ in $X \times Y$. Since $(x^n, y^n)$ is an optimal pair of the control problem
 	\begin{equation}
 		\min_{X\times Y}(J_n + \one_{E_n})
 	\end{equation}
 	for every $n \in \N^+$, then by the definition we obtain that
 	\begin{equation}
 		\min_{X\times Y} (J_n + \one_{E_n}) = (J_n + \one_{E_n})(x^n, y^n).
 	\end{equation}
 	Therefore the sequence $\{(x^n, y^n)\}$ satisfies the condition in Lemma \ref{lemma1}
 	\begin{equation}
 		\liminf_n (J_n + \one_{E_n})(x^n, y^n) = \liminf_n \inf_{X \times Y} (J_n + \one_{E_n}),
 	\end{equation}
 	and using Lemma \ref{lemma1} we finally get that
 	\begin{equation}
 		(J_\infty + \one_{E_\infty})(x_\infty, y_\infty) = \inf_{X\times Y} (J_\infty + \one_{E_\infty}).
 	\end{equation}
\end{proof}

\end{document}